\documentclass[reqno, 12pt,a4letter]{amsart}
\usepackage{amsmath,amsxtra,amssymb,latexsym, amscd,amsthm}
\usepackage{graphicx,color}
\usepackage{subfigure}
\usepackage[utf8]{inputenc}
\usepackage{enumerate}
\usepackage[mathscr]{euscript}
\usepackage{mathrsfs}
\usepackage[english]{babel}
\usepackage[euler]{textgreek}
\usepackage{color}
\usepackage{cite}

\newtheorem {theorem}{Theorem}[section]
\newtheorem {proposition}[theorem]{Proposition}
\newtheorem {lemma}[theorem]{Lemma}

\newtheorem {definition}[theorem]{Definition}
\newtheorem {example}[theorem]{Example}
\newtheorem {remark}[theorem]{Remark}

\def\ar{a\kern-.370em\raise.16ex\hbox{\char95\kern-0.53ex\char'47}\kern.05em}
\def\ees{{\accent"5E e}\kern-.385em\raise.2ex\hbox{\char'23}\kern-.08em}
\def\eex{{\accent"5E e}\kern-.470em\raise.3ex\hbox{\char'176}}
\def\AR{A\kern-.46em\raise.80ex\hbox{\char95\kern-0.53ex\char'47}\kern.13em}
\def\EES{{\accent"5E E}\kern-.5em\raise.8ex\hbox{\char'23 }}
\def\EEX{{\accent"5E E}\kern-.60em\raise.9ex\hbox{\char'176}\kern.1em}
\def\ow{o\kern-.42em\raise.82ex\hbox{
\vrule width .12em height .0ex depth .075ex \kern-0.16em \char'56}\kern-.07em}
\def\OW{O\kern-.460em\raise1.36ex\hbox{
\vrule width .13em height .0ex depth .075ex \kern-0.16em \char'56}\kern-.07em}
\def\UW{U\kern-.42em\raise1.36ex\hbox{
\vrule width .13em height .0ex depth .075ex \kern-0.16em \char'56}\kern-.07em}
\def\DD{D\kern-.7em\raise0.4ex\hbox{\char '55}\kern.33em}

\def\Limsup{\mathop{{\rm Lim}\,{\rm sup}}}
\def\cl{\mbox{\rm cl}\,}
\def\R{\mathbb{R}}

\title{Coderivatives at infinity of set-valued mappings}

\author{DO SANG KIM}
\address[Do Sang Kim]{Department of Applied Mathematics, Pukyong National University, Busan 48513, Republic of Korea}
\email{dskim@pknu.ac.kr}

\author{TI\EES N-S\OW N PH\d{A}M}
\address[Ti\ees n-S\ow n Ph\d{a}m]{Department of Mathematics, Dalat University, 1 Phu Dong Thien Vuong, Dalat, Vietnam}
\email{sonpt@dlu.edu.vn}

\author{NGUYEN MINH TUNG}
\address[Nguyen Minh Tung]{Faculty of Mathematical Economics, Banking University of Ho Chi Minh City, Ho Chi Minh
City, Vietnam}
\email{tungnm@hub.edu.vn}

\author{NGUYEN VAN TUYEN}
\address[Nguyen Van Tuyen]{Department of Mathematics, Hanoi Pedagogical University 2, Xuan Hoa, Phuc Yen, Vinh Phuc, Vietnam}
\email{nguyenvantuyen83@hpu2.edu.vn; tuyensp2@yahoo.com}

\thanks{The first author was supported by the National Research Foundation of Korea Grant funded by the Korean Government (NRF-2019R1A2C1008672)}

\date{\today}



\begin{document}

\maketitle

\begin{abstract} 
In this paper, the concept of coderivatives at infinity of set-valued mappings is introduced.  Well-posedness properties at infinity of set-valued mappings as well as Mordukhovich's criterion at infinity are established. Fermat's rule at infinity in set-valued optimization is also provided. The obtained results, which give new information even in the classical cases of smooth single-valued mappings, provide complete characterizations of the properties under consideration in the setting at infinity of set-valued mappings.
\end{abstract}

\section{Introduction}
The theory of coderivatives (or generalized derivatives) is concerned with the differential properties of set-valued mappings. In the finite dimensional case, this theory associates with a set-valued mapping $F \colon \mathbb{R}^n \rightrightarrows  \mathbb{R}^m$ and a point $(x, y)$ in the graph of $F,$ a set-valued mapping $D^*F({x}, {y}) \colon \R^m \rightrightarrows \R^n,$ which is called the {\em coderivative} of $F$ at $(x, y).$ The theory 
is recognized for many its applications to variational analysis and optimization, as seen for instance in the books \cite{Aubin1990, Borwein2005, Clarke1990, Clarke1998, Ioffe2017, Mordukhovich2006, Mordukhovich2018, Penot2013, Rockafellar1998} with the references therein.

Mordukhovich \cite{Mordukhovich1980} introduced the coderivative mappings $D^*F({x}, {y})$ for the purposes of deriving necessary optimality conditions in optimization problems with nonsmooth constraints and establishing a maximum principle in the optimal control of differential inclusions.  
These mappings enjoy a variety of nice properties, which are important for applications. For example, well-posedness properties of set-valued mappings such as the linear openness (or covering property), the metric regularity and the Lipschitz-like continuity (or Aubin property) can be characterized in terms of coderivatives; also, Fermat's rules (i.e., the first-order necessary optimality conditions) for set-valued optimization problems can be formulated via coderivatives. For more details, see \cite{Mordukhovich2018}.

The aim of this paper is to study the behavior of set-valued mappings in {\em neighborhoods of infinity.} To do this, following the idea of Mordukhovich \cite{Mordukhovich1980}, we define and study {\em coderivatives at infinity} of set-valued mappings. The obtained results, which give new information even in the classical cases of smooth single-valued mappings, are aimed ultimately at applications to diverse problems of variational analysis and optimization following the now-familiar pattern for coderivatives in the classical cases.

At this point, we would like to mention that several novel notions of variational analysis, including tangent/normal cones at infinity to unbounded sets and subdifferentials at infinity for extended-real-value functions, are introduced and studied in the recent papers \cite{PHAMTS2023-4, PHAMTS2023-5}.

\subsection*{Contributions.}
Let $F \colon \mathbb{R}^n \rightrightarrows \mathbb{R}^m$ be a set-valued mapping whose graph is closed, and let $\overline{y} \in \R^m$ be a Jelonek value of $F$, i.e., there exist sequences $x_k \to \infty$ and $y_k \to y$ with $y_k\in F(x_k).$ With notation and definitions given in the next sections, the main contributions of this paper are as follows.

\begin{enumerate}[{\rm $\bullet$}]
\item The coderivative $D^*F (\infty, \overline{y})$ of $F$ at $(\infty, \overline{y})$ is introduced and studied; in particular, basic calculus rules (e.g., sum rule and chain rule) for this coderivative are provided.

\item It will be shown that the following conditions are equivalent:
\begin{enumerate}[{\rm (i)}]
 \item $\mathrm{ker}\,D^*F (\infty, \overline{y}) = \{0\}.$
\item $F$ is linearly open around $(\infty, \overline{y}).$
\item $F$ is metrically regular around $(\infty, \overline{y}).$
\item $F^{-1}$ is Lipschitz-like around $(\overline{y}, \infty).$
\end{enumerate}

\item A version at infinity of Mordukhovich's criterion will be established; namely, the following conditions are equivalent:
\begin{enumerate}[{\rm (i)}]
\item $D^*F (\infty, \overline{y})(0) = \{0\}.$
\item $F$ is Lipschitz-like around $(\infty, \overline{y})$ and $F^{-1}$ is open around $(\overline{y}, \infty).$
\end{enumerate}
(We note that this equivalence is a bit different from the classical one.)

\item Fermat's rule at infinity in set-valued optimization is given. Namely, consider the set-valued optimization problem with explicit geometry constraints
\begin{equation} \label{Problem}
\mathrm{minimize}_{K} \ F (x) \quad \textrm{ subject to } \quad x \in \Omega, \tag{P}
\end{equation}
where $K \subsetneq \mathbb{R}^m$ is a pointed closed convex cone with nonempty interior and $\Omega$ be an unbounded closed subset of $\mathbb{R}^n.$ It will be demonstrated that if $\overline{y} \not \in F(\Omega)$ is a weak efficient value for \eqref{Problem} in which the following constraint qualification at infinity holds
\begin{eqnarray*}
D^*F(\infty, \overline{y})(0) \cap \big(-N_{\Omega}(\infty)\big) &=& \{0\},
\end{eqnarray*}
then there exists a nonzero vector $c^* \in K^{+}$ such that
\begin{eqnarray*}
0 &\in& D^*F(\infty, \overline{y}) (c^*) + N_{\Omega}(\infty),
\end{eqnarray*} 
where $N_{\Omega}(\infty)$ is the {\em normal cone to $\Omega$ at infinity} and $K^+$ is the {\em positive polar} of $K.$
\end{enumerate}

It should be noted that, due to the unboundedness of neighborhoods at infinity, the coderivative at infinity of set-valued mappings enjoy some properties, which are different from the classical ones; see Theorem~\ref{Thm45}.

We confine our study to the finite dimensional case for two reasons. First of all, to lighten the exposition, we would like to concentrate on the basic ideas without technical and notational complications. Furthermore, some of the constructions used and the results obtained are peculiar to finite dimensions.

The rest of the paper is organized as follows. Some definitions and preliminary results from variational analysis are recalled in Section~\ref{Section2}. The notions of normal cone, coderivative and subdifferential at infinity are introduced and studied in Section~\ref{Section3}. 
Well-posedness properties at infinity of set-valued mappings as well as Mordukhovich's criterion at infinity are presented in Section \ref{Section4}. Necessary optimality conditions at infinity for set-valued optimization problems are derived in Section~\ref{Section5}. 

\section{Preliminaries} \label{Section2}

\subsection{Notation and definitions} 
Throughout this work we deal with the Euclidean space $\mathbb{R}^n$ equipped with the usual scalar product $\langle \cdot, \cdot \rangle$ and the corresponding norm $\| \cdot\|.$ We denote by $\mathbb{B}_r(x)$ the closed ball centered at $x$ with radius $r;$  when ${x}$ is the origin of $\mathbb{R}^n$ we write $\mathbb{B}_{r}$ instead of $\mathbb{B}_{r}({x}),$ and when $r = 1$ we write  $\mathbb{B}$ instead of $\mathbb{B}_{1}.$ We will adopt the convention that $\inf \emptyset = +\infty$ and $\sup \emptyset = -\infty.$

For a nonempty set $\Omega \subset \mathbb{R}^n,$ the closure, interior, convex hull and conic hull of $\Omega$ are denoted, respectively, by $\mathrm{cl}\, {\Omega},$ $\mathrm{int}\, {\Omega},$ $ \mathrm{co}\, \Omega$ and $\mathrm{cone}\, \Omega.$

Let $\mathbb{R}_+ := [0, +\infty)$ and $\overline{\mathbb{R}} := \mathbb{R} \cup \{+\infty\}.$ For an extended-real-valued function $f \colon \mathbb{R}^n \rightarrow \overline{\mathbb{R}},$ we denote its {\em effective domain} and {\em epigraph} by,  respectively,
\begin{eqnarray*}
\mathrm{dom} f &:=& \{ x \in \mathbb{R}^n \ | \ f(x) < \infty  \},\\
\mathrm{epi} f &:=& \{ (x, r) \in \mathbb{R}^n \times \mathbb{R} \ | \ f(x) \le r \}.
\end{eqnarray*}
We call $f$ a {\em proper} function if $f(x) < \infty$ for at least one $x \in \mathbb{R}^n,$ or in other words, if $\mathrm{dom} f$ is a nonempty set.
The function $f$ is said to be {\em lower semi-continuous} if for each $x \in \mathbb{R}^n$ the inequality $\liminf_{x' \to {x}} f(x') \ge f({x})$ holds.

Given a nonempty set $\Omega \subset \mathbb{R}^n,$ associate with it the {\em distance function}
\begin{eqnarray*}
\mathrm{dist}(x, \Omega) &:=& \inf_{y \in \Omega} \|x - y\| \quad x \in \mathbb{R}^n,
\end{eqnarray*}
and define the {\em Euclidean projector of} $x \in \mathbb{R}^n$ to $\Omega$ by
\begin{eqnarray*}
\Pi_{\Omega}(x) &:=& \{y \in \Omega \ | \  \|x - y \| = \mathrm{dist}(x, \Omega)\}.
\end{eqnarray*}
If $\Omega = \emptyset,$ we let $\mathrm{dist}(x, \Omega) := \infty$ for all $x.$ The {\em indicator function} $\delta_{\Omega} \colon \mathbb{R}^n \to \overline{\mathbb{R}}$ of the set $\Omega \subset \mathbb{R}^n$ is defined by
\begin{eqnarray*}
\delta_\Omega(x) &:=&
\begin{cases}
0 & \textrm{ if } x \in \Omega, \\
\infty & \textrm{ otherwise.}
\end{cases}
\end{eqnarray*}
By definition, $\Omega$ is closed if and only if $\delta_\Omega$ is lower semi-continuous. The set $\Omega$ is said to be {\em locally closed} at a point ${x}$ (not necessarily in $\Omega$) if $\Omega \cap V$ is closed for some closed neighborhood $V$ of ${x}$ in $\mathbb{R}^n.$ 
We say that $\Omega$ is {\em locally closed} if it is {\em locally closed} at every point $x \in \Omega.$

Let  $F \colon \mathbb{R}^n \rightrightarrows \mathbb{R}^m$ be a set-valued mapping. The {\em domain}, {\em kernel} and {\em graph} of $F$ are taken to be the sets
\begin{eqnarray*}
\mathrm{dom} F &:=& \{x\in \mathbb{R}^n \mid F(x)\not= \emptyset\}, \\
\mathrm{ker} F&:=& \{x \in \mathbb{R}^n \mid 0 \in F(x)\},\\
\mathrm{gph} F&:=& \{(x,y)\in \mathbb{R}^n \times \mathbb{R}^m \mid y \in F(x)\}.
\end{eqnarray*}
The {\em Painlev\'e--Kuratowski outer limit} of $F$ is given by
\begin{eqnarray*}
\Limsup_{x' \to {x}} F(x') &:=& \{y \in \mathbb{R}^m \mid \exists x_k \to {x}, \exists y_k \in F(x_k), y_k \to y\}.
\end{eqnarray*}
The inverse mapping $F^{-1} \colon \mathbb{R}^m \rightrightarrows \mathbb{R}^n$ is defined by $F^{-1}(y) := \{x \in \mathbb{R}^n \mid y \in F(x)\};$ obviously $(F^{-1})^{-1} = F.$
The mapping $F$ is called {\em positively homogeneous} if $0 \in F(0)$ and $t F(x) \subset F(t x)$ for all $x \in \mathbb{R}^n$ and $t > 0;$ in this case, we let
\begin{eqnarray*}
\|F\| &:=& \sup \big \{\|y\| \mid  y \in F(x) \ \ \text{ and }\ \ \|x\|\leq 1\big\}.
\end{eqnarray*}   

\subsection{Normal cones, coderivatives and subdifferentials}

Here we recall some definitions and properties of normal cones to sets, coderivatives of set-valued mappings and subdifferentials of real-valued functions, which can be found in~\cite{Mordukhovich2006, Mordukhovich2018, Rockafellar1998}.

\begin{definition}[normal cones]{\rm 
Consider a set $\Omega\subset\mathbb{R}^n$ and a point ${x} \in \Omega.$
\begin{enumerate}
\item[(i)]  The {\em regular normal cone} (known also as the {\em prenormal} or {\em Fr\'echet normal cone}) $\widehat{N}_{\Omega}({x})$ to $\Omega$ at ${x}$ consists of all vectors $\xi \in\mathbb{R}^n$ satisfying
\begin{eqnarray*}
\langle \xi , x' - {x} \rangle &\le& o(\|x' -  {x}\|) \quad \textrm{ as } \quad x' \to {x} \quad \textrm{ with } \quad x' \in \Omega.
\end{eqnarray*}

\item[(ii)] 
The {\em limiting normal cone} (known also as the {\em basic} or {\em Mordukhovich normal cone}) $N_{\Omega} ({x})$ to $\Omega$ at ${x}$ consists of all vectors $\xi  \in \mathbb{R}^n$ such that there are sequences $x_k \to {x}$ with $x_k \in \Omega$ and $\xi _k \rightarrow \xi $ with $\xi _k \in \widehat N_{\Omega}(x_k),$ or in other words, 
\begin{eqnarray*}
{N}_\Omega({x}) & := &  \Limsup_{x' \xrightarrow{\Omega} {x}}\widehat{N}_\Omega({x}'),
\end{eqnarray*}
where $x' \xrightarrow{\Omega} {x}$ means that $x' \rightarrow {x} $ with $x' \in \Omega.$ 
\end{enumerate}

If $x \not \in \Omega,$ we put $\widehat{N}_{\Omega}({x}) := \emptyset$ and ${N}_\Omega({x}) := \emptyset.$
}\end{definition}

\begin{remark}{\rm
(i) It is well known that $N_\Omega(\overline{x})$ is closed cone (may be non-convex), while $\widehat{N}_\Omega(\overline{x})$ is closed convex cone.

(ii) If $\Omega$ is a manifold of class $C^1,$ then for every point $x \in \Omega,$ the normal cones $\widehat{N}({x}; \Omega)$ and $N({x}; \Omega)$ are equal to the normal space to $\Omega$ at ${x}$ in the sense of differential geometry; see \cite[Example~6.8]{Rockafellar1998}. 

}\end{remark}

\begin{lemma} \label{Lemma2.3}
Let $\Omega \subset \mathbb{R}^n$ be a locally closed set. Then for any $x \in \Omega,$ we have the following relationships
\begin{eqnarray*}
\widehat{N}_{\Omega}(x) &=& \{ \xi \in \mathbb{R}^n \mid \langle \xi, v \rangle \le 0 \quad \textrm{ for all } \quad v \in T_{\Omega}(x) \},\\
N_{\Omega}(x) &=& \Limsup_{x' \rightarrow x} \Big[\mathrm{cone} \big(x' - \Pi_{\Omega}(x') \big) \Big],
\end{eqnarray*}
where $T_{\Omega}(x)$ stands for the {\em contingent cone} to $\Omega$ at $x \in \Omega,$ i.e.,
\begin{eqnarray*}
T_{\Omega}(x) &:=& \Limsup_{t \searrow 0} \frac{\Omega - x}{t}.
\end{eqnarray*}
\end{lemma}

\begin{definition}[coderivatives]{\rm 
Consider a set-valued mapping $F \colon \mathbb{R}^n \rightrightarrows \mathbb{R}^m$ and a point $(\overline{x}, \overline{y}) \in \mathrm{gph} F.$ Assume that $\mathrm{gph} F$ is locally closed at $(\overline{x}, \overline{y}).$ {\em The (basic) coderivative} of $F$ at $(\overline{x}, \overline{y})$ is the set-valued mapping $D^*F(\overline{x}, \overline{y}) \colon \mathbb{R}^m \rightrightarrows \mathbb{R}^n$ defined by
\begin{eqnarray*}
D^*F(\overline{x}, \overline{y})(v) &=& \{u \in \mathbb{R}^n \mid (u, -v)\in N_{\mathrm{gph} F} (\overline{x}, \overline{y}) \} \quad \textrm{ for all } \quad v \in \mathbb{R}^m.
\end{eqnarray*}
}\end{definition}

By definition, the mapping $D^*F(\overline{x}, \overline{y})$ is positively homogeneous.

\begin{definition}[subdifferentials]{\rm
Consider a extended-real-valued function $f\colon\mathbb{R}^n \to \overline{\mathbb{R}}$ and a point ${x} \in \mathrm{dom} f.$
\begin{enumerate}[{\rm (i)}]
\item The {\em regular} (or {\em Fr\'echet}) {\em subdifferential} of $f$ at ${x}$ is 
$$\widehat{\partial}f({x}) :=\{ u \in \mathbb{R}^n \mid  (u, -1) \in \widehat{N}_{\mathrm{epi} f}({x},f({x}))  \}.  $$
\item The {\em limiting} (or {\em Mordukhovich}) {\em subdifferential} of $f$ at ${x}$ is 
$$\partial f({x}):=\{ u \in \mathbb{R}^n \mid  (u, -1)\in {N}_{\mathrm{epi} f}({x},f({x}))  \}.$$
\item The {\em singular subdifferential} of $f$ at ${x}$ is 
$$\partial^\infty f({x}) := \{ u \in \mathbb{R}^n \mid (u, 0)\in {N}_{\mathrm{epi} f}({x},f({x}))  \}.$$
\end{enumerate}
If $x \not \in \mathrm{dom} f,$ we put $\widehat{\partial}f({x}) := \emptyset,$ $\partial f({x}) := \emptyset$, and $\partial^\infty f({x}) := \emptyset.$ 
}\end{definition}

\subsubsection{Known facts}
The following four lemmas are well known; see \cite{Mordukhovich2006, Mordukhovich2018, Rockafellar1998}.

\begin{lemma} \label{Lemma2.6}
For any closed set $\Omega \subset \mathbb{R}^n$ and point ${x} \in \Omega,$ we have
\begin{eqnarray*}
\partial \delta_\Omega({x}) &=&  \partial^\infty \delta_\Omega({x})  \ = \ N_{\Omega}({x}).
\end{eqnarray*}
\end{lemma}

\begin{lemma} \label{Lemma2.7}
For any point $\overline{x} \in \mathbb{R}^n,$ we have
\begin{eqnarray*}
\partial \|\cdot - \overline{x}\| (x) &=& 
\begin{cases}
\mathbb{B} & \quad \textrm{ if } x = \overline{x}, \\
\frac{x - \overline{x}}{\|x - \overline{x}\| } & \quad \textrm{ otherwise.}
\end{cases}
\end{eqnarray*}
\end{lemma}

\begin{lemma}[Fermat rule] \label{Lemma2.8}
If a function $f \colon \mathbb{R}^n \to \overline{\mathbb{R}}$ has a local minimum at $\overline{x} \in \mathrm{dom} f,$ then $0 \in \partial f(\overline{x}).$
\end{lemma}

\begin{lemma}\label{Lemma2.9} 
Let $f_i \colon \mathbb{R}^n \to \overline{\mathbb{R}}$, $i=1,\dots,m$ with $m \geq 2$, be lower semi-continuous at $\overline{x} \in \mathbb{R}^n$  and let all but one of these functions be Lipschitz around $\overline{x}.$ Then the following inclusion holds:
\begin{eqnarray*}
\partial \left( f_1+\cdots+f_m\right)(\overline{x}) & \subset & \partial f_1(\overline{x})+\cdots+\partial f_m(\overline{x}).
\end{eqnarray*}
\end{lemma}

We also recall the Ekeland variational principle (see \cite{Ekeland1974, Ekeland1979}).

\begin{theorem}\label{Theorem210}
Let $(X, d)$ be a complete metric space, and $f \colon (X, d) \to \overline{\mathbb{R}}$ be a proper, lower semi-continuous and bounded from below function. Let $\epsilon >0$ and $x_0 \in X$ be given such that 
\begin{eqnarray*}
f (x_0) &\le& \inf_{x \in X} f (x) + \epsilon.
\end{eqnarray*}
Then, for any $\lambda >0$ there is a point $x_1 \in X$ satisfying the following conditions
\begin{enumerate}[{\rm (i)}]
\item $f (x_1) \le f(x_0),$
\item $d(x_1, x_0) \le \lambda,$ and
\item $f (x_1) \le  f(x) + \frac{\epsilon}{\lambda}d(x, x_1)$ for all $x \in X.$
\end{enumerate}
\end{theorem}

\section{Normals, coderivatives and subdifferentials at infinity}\label{Section3}

\subsection{Normal cones at infinity}

In our geometric approach to coderivatives, we start with constructing normal cones to nonempty subsets of finite-dimensional spaces. 
To see this, let $\Omega$ be a subset of $\mathbb{R}^n \times \mathbb{R}^m$ and let $\overline{y} \in \mathbb{R}^m.$ 
Suppose that $\Omega$ is {\em locally closed} at $(\infty, \overline{y}),$ i.e., there are a constant $R > 0$ and a closed neighborhood $V$ of $\overline{y}$ in $\mathbb{R}^m$ such that the set
\begin{eqnarray*}
\{(x, y) \in \mathbb{R}^n \times \mathbb{R}^m \mid \|x\| \ge R \textrm{ and } y \in V\}
\end{eqnarray*}
is closed. Obviously, if the set $\Omega$ is closed then it is locally closed at $(\infty, \overline{y})$ for every $\overline{y} \in \mathbb{R}^m.$

The following definition is inspired by the works of Mordukhovich \cite{Mordukhovich2006, Mordukhovich2018}; see also the recent papers \cite{PHAMTS2023-4, PHAMTS2023-5}.

\begin{definition}{\rm 
The {\em normal cone to the set $\Omega$ at $(\infty, \overline{y})$} is defined by
\begin{eqnarray*}
N_{\Omega}(\infty, \overline{y}) &:=& \Limsup_{(x, y) \xrightarrow{\Omega} (\infty, \overline{y})} \widehat{N}_{\Omega}(x, y),
\end{eqnarray*}
where ${(x, y) \xrightarrow{\Omega} (\infty, \overline{y})} $ means that $\|x\| \rightarrow \infty$ and $y \rightarrow \overline{y}$ with $(x, y) \in \Omega.$ 
}\end{definition}

By definition, it is not hard to see that $N_{\Omega}(\infty, \overline{y})$ is a (possibly empty) closed cone.
\begin{example}{\rm 
(i) Let $\Omega :=\{(x,y) \in \mathbb{R}^2 \mid x \geq y \}.$ For $(x, y) \in \Omega$ we have 
\begin{eqnarray*}
\widehat{N}_{\Omega}(x,y) &=& N_{\Omega}(x,y) \ = \ 
\begin{cases}
\{(0,0)\} & \textrm{ if } x > y, \\
\{(t, -t) \mid t \geq 0\} & \textrm{ if } x=y.
\end{cases}
\end{eqnarray*}
Therefore, for any $\overline{y} \in \mathbb{R}$ we have 
\begin{eqnarray*}
N_{\Omega}(\infty, \overline{y}) &=& \Limsup_{(x,y) \xrightarrow{\Omega} (\infty,\overline{y})} \widehat{N}_{\Omega}(x,y) \ = \ \{(0,0)\}.
\end{eqnarray*}

(ii) Let $\Omega :=\{(x,y) \in \mathbb{R}^2 \mid y \geq e^{x}\}.$ For $(x,y) \in \Omega$ we have 
\begin{eqnarray*}
\widehat{N}_{\Omega}(x,y) &=& N_{\Omega}(x,y) \ = \ 
\begin{cases}
\{(0,0)\} & \textrm{ if } y >e^{x}, \\
\{(te^{x}, -t) \mid t \geq 0\} & \textrm{ if } y=e^{x}.
\end{cases}
\end{eqnarray*}
Therefore, for $\overline{y} \in \mathbb{R}$ we have 
\begin{eqnarray*}
N_{\Omega}(\infty,\overline{y}) &=& \Limsup_{(x,y) \xrightarrow{\Omega} (\infty,\overline{y})} \widehat{N}_{\Omega}(x,y) \ = \
\begin{cases}
\emptyset & \textrm{ if } \overline{y}<0, \\
\{0\}\times (-\infty, 0] & \textrm{ if } \overline{y}=0,\\
\{(0,0)\}& \textrm{ if } \overline{y}>0. 
\end{cases}
\end{eqnarray*}
}\end{example}

\begin{proposition}\label{Prop33}
The following equality holds:
\begin{eqnarray*}
N_{\Omega}(\infty, \overline{y}) &=& \Limsup_{(x, y) \xrightarrow{\Omega} (\infty, \overline{y})} N_{\Omega}(x, y).
\end{eqnarray*}
\end{proposition}

\begin{proof}
Indeed, we know that $\widehat{N}_{\Omega}(x, y) \subset N_{\Omega}(x, y)$ for all $(x, y) \in \mathbb{R}^n \times \mathbb{R}^m.$ By definition, hence
\begin{eqnarray*}
N_{\Omega }(\infty, \overline{y})  &=& \Limsup_ {(x, y) \xrightarrow{\Omega} (\infty, \overline{y})}  \widehat{N}_{\Omega}(x, y) \ \subset \ \Limsup_{(x, y) \xrightarrow{\Omega} (\infty, \overline{y})}  N_{\Omega}(x, y).
\end{eqnarray*}
For the converse, take any $\xi \in \Limsup_{(x, y) \xrightarrow{\Omega} (\infty, \overline{y})}  N_{\Omega}(x, y),$ which means that there exist sequences $(x_k, y_k) \in \Omega$ and $\xi_k \in N_{\Omega}(x_k, y_k) $ such that $(x_k, y_k) \to (\infty, \overline{y})$ and $\xi_k \to \xi.$ By definition, for each $k > 0,$ there are $(x_k', y_k') \in \Omega$ and $\xi_k' \in \widehat{N}_{\Omega}(x_k', y_k') $ such that $\|(x_k', y_k') - (x_k, y_k)\| < \frac{1}{k}$ and $\|\xi_k' - \xi_k \| < \frac{1}{k}.$ Clearly, $(x_k', y_k') \xrightarrow{\Omega} (\infty, \overline{y})$ and $\xi_k' \to \xi$ as $k \to \infty.$ By definition, therefore $\xi \in {N}_{\Omega}(\infty, \overline{y}).$
\end{proof}

\begin{proposition}
Let $\Omega_1, \Omega_2 \subset \mathbb{R}^n \times \mathbb{R}^m$ be locally closed at $(\infty, \overline{y})$ such that
\begin{eqnarray*}
N_{\Omega_1}(\infty, \overline{y}) \cap \big (-N_{\Omega_2}(\infty, \overline{y}) \big) &=& \{0\}.
\end{eqnarray*}
Then $\Omega_1 \cap \Omega_2$ is locally closed at $(\infty, \overline{y}),$ and
\begin{eqnarray*}
N_{\Omega_1 \cap \Omega_2}(\infty, \overline{y})  &\subset& N_{\Omega_1}(\infty, \overline{y}) + N_{\Omega_2}(\infty, \overline{y}).
\end{eqnarray*}
\end{proposition}

\begin{proof}
We first show that there exist a constant $R > 0$ and a neighborhood $V$ of $\overline{y}$ in $\mathbb{R}^m$ such that
\begin{eqnarray*}
N_{\Omega_1}(x, y) \cap \big(-N_{\Omega_2}(x, y) \big) &=& \{0\}
\end{eqnarray*}
for all $(x, y) \in \Omega_1 \cap \Omega_2$ with $\|x\| \ge R$ and $y \in V.$

Indeed, if this were not true, there would exist sequences $(x_k, y_k) \in \Omega_1 \cap \Omega_2$ and $\xi_k \in N_{\Omega_1}(x_k, y_k) \cap (-N_{\Omega_2}(x_k, y_k)$ such that $\|x_k\| \to \infty, y_k \to \overline{y}$ and $\xi_k \ne 0.$ Since $N_{\Omega_1}(x_k, y_k)$ and $N_{\Omega_2}(x_k, y_k)$ are cones, we can assume that $\|\xi_k\| = 1$ for all $k.$ Passing to a subsequence if necessary we may assume that the sequence $\xi_k$ converges to some $\xi.$ Clearly, $\|\xi \| = 1$ and  $\xi \in N_{\Omega_1}(\infty, \overline{y}) \cap \big(-N_{\Omega_2}(\infty, \overline{y}) \big),$ in contradiction to our assumption.

We now prove the desired inclusion. To this end, take any $\xi \in N_{\Omega_1 \cap \Omega_2}(\infty, \overline{y}).$ By Proposition~\ref{Prop33}, there exist sequences $x_k \in \Omega_1 \cap \Omega_2$ and $\xi_k \in N_{\Omega_1 \cap \Omega_2}(x_k, y_k)$ such that $\|x_k\| \to \infty, y_k \to \overline{y}$ and $\xi_k \to \xi.$ Then for all $k$ sufficiently large, we have
\begin{eqnarray*}
N_{\Omega_1}(x_k, y_k) \cap \big(-N_{\Omega_2}(x_k, y_k) \big) &=& \{0\}.
\end{eqnarray*}
It follows from \cite[Theorem~2.16]{Mordukhovich2018} that
\begin{eqnarray*}
N_{\Omega_1 \cap \Omega_2}(x_k, y_k)  &\subset& N_{\Omega_1}(x_k, y_k) + N_{\Omega_2}(x_k, y_k).
\end{eqnarray*}
Consequently, we can write $\xi_k = u_k + v_k$ for some $u_k \in N_{\Omega_1}(x_k, y_k)$ and $v_k \in N_{\Omega_2}(x_k, y_k).$ There are two cases to be considered.

\subsubsection*{Case 1: the sequence $u_k$ is bounded.} 
Then the sequence $v_k$ is bounded too. Passing to subsequences if necessary we may assume that $u_k$ and $v_k$ converge to some $u$ and $v,$ respectively. Clearly, $\xi = u + v,$ and by Proposition~\ref{Prop33}, we have $u \in N_{\Omega_1}(\infty, \overline{y}) $ and $v \in N_{\Omega_2}(\infty, \overline{y}).$ 
Therefore $\xi  \in  N_{\Omega_1}(\infty, \overline{y})  +  N_{\Omega_2}(\infty, \overline{y}).$

\subsubsection*{Case 2: the sequence $u_k$ is unbounded.} 
Then the sequence $v_k$ is unbounded too. Since $u_k + v_k$ is a convergent sequence, it is easy to see that the sequence $\frac{u_k}{\|v_k\|}$ is bounded. Passing to subsequences if necessary we may assume that the sequences $\frac{u_k}{\|v_k\|}$ and $\frac{v_k}{\|v_k\|}$ converge to some $u$ and $v,$ respectively. Clearly, $u \in  N_{\Omega_1}(\infty, \overline{y}), v \in N_{\Omega_2}(\infty, \overline{y})$ and $u = - v \ne 0,$ in contradiction to our assumption.
\end{proof}

\subsection{Coderivatives at infinity}
In this subsection, the {\em coderivatives at infinity} of set-valued mappings will be defined in terms of the normal cones at infinity to the graphs of the mappings in question. Namely, let $F\colon \mathbb{R}^n\rightrightarrows \mathbb{R}^m$ be a set-valued mapping. We will associate with $F$ the {\em Jelonek set}
\begin{eqnarray*}
J(F) &:=& \{y \in \mathbb{R}^m \mid \exists (x_k, y_k)  \xrightarrow{\mathrm{gph} F} (\infty, y) \}.
\end{eqnarray*}
Let $\overline{y} \in J(F),$ and suppose in this subsection that $\mathrm{gph}F$ is locally closed at $(\infty, \overline{y}).$
 
\begin{definition}{\rm 
{\em The coderivative} of $F$ at $(\infty, \overline{y})$ is the set-valued mapping $D^*F(\infty, \overline{y})\colon \mathbb{R}^m \rightrightarrows \mathbb{R}^n$ defined by
\begin{eqnarray*}
D^*F(\infty, \overline{y})(v) &:=& \{u\in\mathbb{R}^n\;|\; (u,-v)\in N_{\mathrm{gph} F}(\infty, \overline{y})\} \quad \textrm{ for } \ v \in \mathbb{R}^m.
\end{eqnarray*}
}\end{definition}

By definition, it is easy to see that the coderivative of $F$ at $(\infty, \overline{y})$ is a closed, positively homogeneous mapping.

\begin{example}\label{VD36} {\rm 
(i) Consider the set-valued mapping $F \colon \mathbb{R} \rightrightarrows \mathbb{R}$ defined by
\begin{eqnarray*}
F(x) &:=&
\begin{cases}
\{1\} & \textrm{ if } x > 0, \\
[-1, 1] & \textrm{ if } x = 0, \\
\{-1\} & \textrm{ if } x < 0, \\
\end{cases}
\end{eqnarray*}
A direct calculation shows that $J(F)=\{-1, 1\},$ and for $\overline{y} \in J(F)$ we have
\begin{eqnarray*}
N_{\mathrm{gph} F}(\infty, \overline{y}) &=&  \{(0, v) \mid v \in \mathbb{R}\},
\end{eqnarray*}
which yields $D^*F(\infty, \overline{y})(v) =\{0\}$ for all $v \in \mathbb{R}.$

(ii) Consider the set-valued mapping $F \colon \mathbb{R} \rightrightarrows \mathbb{R}$ defined by
\begin{eqnarray*}
F(x) &:=& 
\begin{cases}
[0, x] & \textrm{ if } x \ge 0, \\
x^2 & \textrm{ otherwise.}
\end{cases}
\end{eqnarray*}
A direct calculation shows that $J(F) = [0, +\infty)$ and 
\begin{eqnarray*}
N_{\mathrm{gph} F}(\infty, \overline{y}) &=& 
\begin{cases}
\{(0, v)\mid v \le 0\} & \textrm{ if } \overline{y} = 0, \\
\{0, 0\} & \textrm{ if } \overline{y} > 0. 
\end{cases}
\end{eqnarray*}
Hence
\begin{eqnarray*}
D^*F(\infty, 0)(v) &=& 
\begin{cases}
\{0\} & \textrm{ if } v \ge 0, \\
\emptyset & \textrm{ otherwise,}
\end{cases}
\end{eqnarray*}
and if $\overline{y} > 0,$ then
\begin{eqnarray*}
D^*F(\infty, \overline{y})(v) &=& 
\begin{cases}
\{0\} & \textrm{ if } v = 0, \\
\emptyset & \textrm{ otherwise.} 
\end{cases}
\end{eqnarray*}

(iii) Consider the set-valued mapping $F \colon \mathbb{R}^2 \rightrightarrows \mathbb{R}$ defined by
$F(x_1,x_2) := \{x_1\}.$ A direct calculation shows that $J(F) = \mathbb{R},$ and for $\overline{y} \in J(F)$ we have
\begin{eqnarray*}
N_{\mathrm{gph} F}(\infty, \overline{y}) &=& \{ (v, 0, -v) \mid v \in \mathbb{R} \},
\end{eqnarray*}
which yields $D^*F(\infty, \overline{y})(v) = \{(v,0)\}$ for all $v \in \mathbb{R}.$

(iv) Consider an unbounded closed set $\Omega \subset \mathbb{R}^n$ and define the {\em indicator mapping} $\Delta_{\Omega} \colon \mathbb{R}^n \rightrightarrows \mathbb{R}^m$ of $\Omega$ by
\begin{eqnarray*}
\Delta_{\Omega} (x) &:=&
\begin{cases}
\{0\} &\textrm{ if } x \in \Omega,\\
\emptyset &\textrm{ otherwise.}
\end{cases}
\end{eqnarray*}
By definition, we have $J(\Delta_{\Omega}) = \{0\}$ and $\mathrm{gph}\Delta_{\Omega} = \Omega \times \{0\}.$ Hence
\begin{eqnarray*}
N_{\mathrm{gph} \Delta_{\Omega}}(\infty, 0) &=& \Limsup_{(x, y)\stackrel{\mathrm{gph} \Delta_{\Omega}}\longrightarrow(\infty, 0)}\widehat{N}_{\mathrm{gph} \Delta_{\Omega}}(x,y)\\
&=& \Limsup_{(x, y)\stackrel{\mathrm{gph} \Delta_{\Omega}}\longrightarrow(\infty, 0)}\widehat{N}_{\Omega}(x) \times N_{\{0\}}(0)\\
&=& \Limsup_{x \xrightarrow{\Omega} \infty} \widehat{N}_{\Omega}(x) \times \mathbb{R}^m,
\end{eqnarray*}
which yields
\begin{eqnarray*}
D^* \Delta_{\Omega}(\infty, 0)(v) &=& \Limsup_{x \xrightarrow{\Omega} \infty} \widehat{N}_{\Omega}(x) \quad \textrm{ for all } \quad v \in \mathbb{R}^m.
\end{eqnarray*}
}\end{example}

Let us present the following property, which easily follows from Proposition~\ref{Prop33}.

\begin{proposition}\label{Prop37}
Let $\overline{y} \in J(F).$ For all ${v} \in \mathbb{R}^m$ one has
\begin{eqnarray*}
D^*F(\infty, \overline{y})({v}) &=& \mathop{\mathop {\Limsup}\limits_{(x, y)\xrightarrow{\mathrm{gph} F}(\infty, \overline{y})} }
\limits_{v'  \to {v}} {D^*F(x, y)(v')}.
\end{eqnarray*}
\end{proposition}

Next we present some calculus rules for coderivatives of set-valued mappings.
Recall that the sum of two set-valued mappings $F_1, F_2 \colon \mathbb{R}^n \rightrightarrows {\mathbb{R}^m}$ is defined by
\begin{eqnarray*}
(F_1+F_2)(x) &:=& \{ y_1+y_2 \mid y_1 \in F_1(x), y_2 \in F_2(x)\} \quad \textrm{ for all } \; x\in \mathbb{R}^n.
\end{eqnarray*}
For each $y \in \mathbb{R}^m,$ we let
\begin{eqnarray*}
J^+(F_1 + F_2, y) \ := \ \{ (y_1,y_2)\in \mathbb{R}^{m} \times \mathbb{R}^m &\mid & \exists x_k \rightarrow \infty, \
\exists  y_{1, k} \in F_1(x_k), \ \exists  y_{2, k} \in F_2(x_k), \\
&& (y_{1, k},y_{2, k})\to (y_1,y_2), \ y_1+y_2 = y\}. 
\end{eqnarray*}

\begin{theorem}[sum rule] \label{SumeRule}
Suppose $F := F_1 + F_2$ for closed-graph set-valued mappings $F_1, F_2 \colon \mathbb{R}^n \rightrightarrows \mathbb{R}^m$ and let $\overline{y} \in J(F).$ Assume both of the following conditions are satisfied.
\begin{enumerate}[{\rm (a)}]
\item {\rm (boundedness condition):} There exist constants $R > 0, \rho > 0$ and a neighborhood $V$ of $\overline{y}$ such that whenever $\|x\|  > R,$ $y_i \in F_i(x)$ and $y_1 + y_2 \in V,$ one has $\|y_i\| < \rho$ for $i = 1, 2.$

\item {\rm (constraint qualification):} For every $(y_1,y_2)\in J^+(F_1+F_2,\overline{y})$ we have
\begin{eqnarray*}
D^*F_1(\infty , y_1) (0)\cap \big(-D^*F_2(\infty , y_2)(0) \big) &=& \{0\}.
\end{eqnarray*}

\end{enumerate}
Then $\mathrm{gph}F$ is locally closed at $(\infty, \overline{y}),$ and for all $v \in \mathbb{R}^m$ we have
\begin{eqnarray*}
D^*F (\infty, \overline{y})(v) &\subset& \bigcup_{(\overline{y}_1,\overline{y}_2) \in J^+(F_1+F_2,\overline{y})}  D^*F_1(\infty , \overline{y}_1) (v)+ D^*F_2(\infty , \overline{y}_2)(v). 
\end{eqnarray*}
\end{theorem}

\begin{proof}
The condition~(a), together with the closedness of the graphs of the mappings $F_1$ and $F_2,$ ensures that $\mathrm{gph}F$ is locally closed at $(\infty, \overline{y}).$ 

Next observe that, by increasing $R$ and shrinking $V$ if necessary, we may assume 
\begin{eqnarray} \label{3PT1}
D^*F_1(x , y_1) (0) \cap \big(- D^*F_2(x, y_2)(0) \big) &=& \{0\}
\end{eqnarray}
for all $x \in \mathbb{R}^n \setminus \mathbb{B}_R,$ $y_1 \in F_1(x)$ and $y_2 \in F_2(x)$ with $y_1 + y_2 \in V.$
(Indeed, if this is not true, there are sequences $x_k \in \mathbb{R}^n, y_{1, k} \in F_1(x_k), y_{2, k} \in F_2(x_k)$ and $u_k \in \mathbb{R}^n \setminus \{0\}$ such that 
$\|x_k\| \to \infty,$ $y_{1, k} + y_{2, k} \to \overline{y}$ and 
\begin{eqnarray*}
u_k &\in &  D^*F_1(x_k, y_{1, k}) (0) \cap \big (- D^*F_2(x_k , y_{2, k})(0) \big).
\end{eqnarray*}
This, together with the positive homogeneity of coderivative mappings, yields
\begin{eqnarray*}
\dfrac{u_k}{\|u_k\|} &\in &  D^*F_1(x_k, y_{1, k}) (0) \cap \big (- D^*F_2(x_k , y_{2, k}) (0) \big).
\end{eqnarray*}
Passing to subsequences if necessary and applying Proposition~\ref{Prop37}, we get $(\overline{y}_1, \overline{y}_2) \in J^+(F_1 + F_2, \overline{y})$
(which comes from the condition~(a)) and a unit vector $u$ satisfying
\begin{eqnarray*}
u &\in&  D^*F_1(\infty, \overline{y}_1) (0) \cap \big (-  D^*F_2(\infty ,\overline{y}_2)(0) \big),
\end{eqnarray*}
which contradicts the condition~(b).)

Take any $u \in D^*F(\infty, \overline{y})(v) $. In view of Proposition~\ref{Prop37}, we find sequences $(x_k, y_k) \in \mathrm{gph} F,$ $u_k \in  \mathbb{R}^n$ and $v_k \in \mathbb{R}^m$ with $u_k \in D^*F(x_k,y_k)(v_k)$ such that
\begin{eqnarray*}
x_k \to \infty, \; y_k \to \overline{y}, \; u_k \to u \ \textrm{ and } \ v_k \to v. 
\end{eqnarray*}
By definition, there are $y_{1, k} \in F_1(x_k)$ and $y_{2, k} \in F_2(x_k)$ such that $y_k = y_{1,k} + y_{2, k}.$ By the condition~(a), we may assume that $y_{1, k}$ and $y_{2, k}$ converge to some $\overline{y}_1$ and $\overline{y}_2,$ respectively. Clearly, $(\overline{y}_1,\overline{y}_2) \in J^+(F_1 + F_2, \overline{y}).$

Now putting~\eqref{3PT1} together with \cite[Theorem~3.9]{Mordukhovich2018} (see also \cite[Theorem~10.41]{Rockafellar1998}), for all $k$ sufficiently large we can write
\begin{eqnarray*}
u_k &=& u_{1, k} + u_{2, k},
\end{eqnarray*}
where $u_{1, k} \in D^*F_1(x_k, y_{1, k})(v_k)$ and $u_{2, k} \in D^*F_2(x_k, y_{2, k})(v_k)$. We consider two following cases. 

\subsubsection*{Case 1: the sequence $u_{1, k}$ is bounded.} 
Then the sequence $u_{2, k}$ is bounded too. Passing to subsequences if necessary we may assume that the sequences $u_{1, k}$ and $u_{2, k}$ converge to some $u_1$ and $u_2,$ respectively. By Proposition~\ref{Prop37}, $u_1 \in  D^*F_1(\infty , \overline{y}_1) (v) $ and $u_2 \in  D^*F_2(\infty , \overline{y}_2) (v).$ Therefore,  $u  = u_1 + u_2 \in   D^*F_1(\infty , \overline{y}_1) (v) + D^*F_2(\infty , \overline{y}_2)(v),$ as required.

\subsubsection*{Case 2: the sequence $u_{1, k}$ is unbounded.} 
Then the sequence $u_{2, k}$ is unbounded too. By the positive homogeneity of coderivative mappings, we have
\begin{eqnarray*}
\dfrac{u_{1, k}}{\|u_{1, k}\|} &\in& D^*F_1(x_k, y_{1, k}) \left(\dfrac{v_k}{\|u_{1, k}\|} \right) \quad \textrm{ and } \quad  
\dfrac{u_{2, k}}{\|u_{1, k}\|} \ \in \ D^*F_2(x_k, y_{2, k}) \left(\dfrac{v_k}{\|u_{1, k}\|} \right).
\end{eqnarray*}
Since $u_k = u_{1, k} + u_{2,k}$ is a convergent sequence, it is easy to see that the sequence $\frac{u_{2, k}}{\|u_{1, k}\|}$ is bounded. Passing to subsequences if necessary we may assume that the sequences $\frac{u_{1, k}}{\|u_{1, k}\|}$ and $\frac{u_{2, k}}{\|u_{1, k}\|}$ converge to some $u_1'$ and $u_2',$ respectively. Clearly, $\|u_1'\| = 1$ and $u_1' + u_2' = 0;$ moreover, by Proposition~\ref{Prop37}, we get $u_1' \in D^*F_1(\infty , \overline{y}_1) (0) $ and $u_2' \in D^*F_2(\infty , \overline{y}_2) (0)$. These contradict our constraint qualification.
\end{proof}

Composition of $F_1\colon \mathbb{R}^n \rightrightarrows \mathbb{R}^p$ with $F_2 \colon \mathbb{R}^p \rightrightarrows \mathbb{R}^m$ is defined by
\begin{eqnarray*}
F_2 \circ F_1(x) &:=& \{ y \in \mathbb{R}^m \mid F_1(x) \cap F_2^{-1}(y) \ne \emptyset \} \quad \textrm{ for all } \; x \in \mathbb{R}^n. 
\end{eqnarray*}
to get a set-valued mapping $F_2 \circ F_1 \colon \mathbb{R}^n \rightrightarrows \mathbb{R}^m.$ For each $y \in \mathbb{R}^m$, we let
\begin{eqnarray*}
J^\circ(F_2\circ F_1,y) &:=&  \{ z \in \mathbb{R}^p \mid \exists x_k \rightarrow \infty, \ \exists z_{k} \in F_1(x_k), \ \exists y_{k} \in F_2(z_k), \ z_{k} \to z, \ y_{k} \to y\}. 
\end{eqnarray*}

\begin{theorem}[chain rule]
Suppose $F := F_2 \circ F_1$ for  closed-graph set-valued mappings $F_1\colon \mathbb{R}^n \rightrightarrows {\mathbb{R}^p}$ and $F_2 \colon \mathbb{R}^p \rightrightarrows {\mathbb{R}^m},$ and let $\overline{y} \in J(F).$ Assume both of the following conditions are satisfied.
\begin{enumerate}[{\rm (a)}]
\item {\rm (boundedness condition):} There exist constants $R > 0, R' > 0$ and a neighborhood $V$ of $\overline{y}$ in $\mathbb{R}^m$ such that
for all $x \in \mathbb{R}^n \setminus \mathbb{B}_R$ and $y \in V$ we have
\begin{eqnarray*}
F_1(x) \cap F_2^{-1}(y) &\subset& \mathbb{B}_{R'}.
\end{eqnarray*}

\item {\rm (constraint qualification):} For every $z \in J^\circ(F_2 \circ F_1, \overline{y}),$ one has 
\begin{eqnarray*}
D^*F_2(z,\overline{y}) (0)\cap \mathrm{ker} D^*F_1(\infty , z) &=& \{0\}.
\end{eqnarray*}
\end{enumerate}
Then $\mathrm{gph}F$ is locally closed at $(\infty, \overline{y}),$ and 
\begin{eqnarray*}
D^*(F_2 \circ F_1)(\infty, \overline{y}) &\subset& \bigcup_{\overline{z} \in J^\circ(F_2 \circ F_1, \overline{y})}  D^*F_1(\infty , \overline{z}) \circ D^*F_2(\overline{z} , \overline{y}). 
\end{eqnarray*}
\end{theorem}

\begin{proof}
The condition~(a), together with the closedness of the graphs of the mappings $F_1$ and $F_2,$ implies that $\mathrm{gph}F$ is locally closed at $(\infty, \overline{y}).$ 

Next observe that, by increasing $R$ and shrinking $V$ if necessary, we may assume 
\begin{eqnarray}\label{3PT2}
D^*F_2(z, y) (0) \cap \mathrm{ker} D^*F_1(x, z) &=& \{0\}
\end{eqnarray}
for all $x \in \mathbb{R}^n \setminus \mathbb{B}_R, y \in V$ and $z \in F_1(x) \cap F_2^{-1}(y).$
(Indeed, if this is not true, there are sequences $(x_k, y_k) \to (\infty, \overline{y}), z_k \in  F_1(x_k) \cap F_2^{-1}(y_k)$ and $u_k \in \mathbb{R}^n \setminus \{0\}$ such that 
\begin{eqnarray*}
u_k &\in & D^*F_2(z_k, y_k) (0) \cap \mathrm{ker} D^*F_1(x_k, z_k).
\end{eqnarray*}
This, together with the positive homogeneity of coderivative mappings, yields
\begin{eqnarray*}
\dfrac{u_k}{\|u_k\|} &\in &  D^*F_2(z_k, y_k) (0) \cap \mathrm{ker} D^*F_1(x_k, z_k).
\end{eqnarray*}
Passing to subsequences if necessary and applying Proposition~\ref{Prop37}, we get $\overline{z} \in J^\circ(F_2 \circ F_1, \overline{y})$ 
(which comes from the condition~(a)) and a unit vector $u$ satisfying
\begin{eqnarray*}
u &\in&  D^*F_2(\overline{z}, \overline{y}) (0) \cap \mathrm{ker} D^*F_1(\infty, \overline{z}),
\end{eqnarray*}
which contradicts the condition~(b).)

Let $v \in \mathbb{R}^m,$ and take any $u \in D^* (F_2 \circ F_1)(\infty, \overline{y})(v).$ By Proposition~\ref{Prop37}, there are sequences $x_k \in \mathbb{R}^n,$ 
$y_k \in (F_2 \circ F_1) (x_k),$ $v_k \in \mathbb{R}^m$ and $u_k \in  D^* (F_2 \circ F_1)(x_k, y_k)(v_k)$ satisfying 
\begin{eqnarray*}
\|x_k\| \to \infty, \ y_k \to  \overline{y}, \ v_k \to v, \quad \textrm{and} \quad u_k \to u.
\end{eqnarray*}
As $y_k \in (F_2 \circ F_1) (x_k),$ we find $z_k \in F_1(x_k)$ such that $y_k \in F_2(z_k)$.
By the condition~(a) and passing to a subsequence if necessary, we may assume that the sequence $z_k$ converges to some $\overline{z}.$
Clearly $\overline{z} \in J^\circ(F_2 \circ F_1, \overline{y}).$ 

Putting the relation \eqref{3PT2} together with \cite[Theorem~3.11]{Mordukhovich2018} (see also \cite[Theorem~10.37]{Rockafellar1998}), we obtain for all $k$ large enough
\begin{eqnarray*}
D^* (F_2 \circ F_1)(x_k, y_k)(v_k) &\subset & D^* F_1(x_k, z_k) \circ D^* F_2(z_k,y_k)(v_k),
\end{eqnarray*}
and so there is $\widehat{u}_{k} \in D^* F_2(z_k,y_k)(v_k)$ such that $u_{k} \in D^* F_1(x_k, z_k)(\widehat{u}_k). $

Passing to subsequences if necessary, it suffices to consider the following two cases.

\subsubsection*{Case 1: the sequence $\widehat{u}_k$ is bounded.} 
We may assume further that $\widehat{u}_k$ converges to some $\widehat{u}.$ By Proposition~\ref{Prop37}, then $\widehat{u} \in  D^*F_2(\overline{z} , \overline{y}) (v) $ and $u \in  D^*F_1(\infty , \overline{z}) (\widehat{u}),$ as required. 

\subsubsection*{Case 2: the sequence $\widehat{u}_k$ is unbounded.} 
By the positive homogeneity of coderivative mappings, we see that 
\begin{eqnarray*}
\dfrac{\widehat{u}_k}{\|\widehat{u}_k\|} \in D^* F_2(z_k,y_k) \left( \dfrac{v_k}{\|\widehat{u}_k\|}\right) \quad \textrm{and} \quad \dfrac{u_{k}}{\|\widehat{u}_k\|} \in D^* F_1(x_k,z_k) \left(\dfrac{\widehat{u}_k}{\|\widehat{u}_k\|}\right). 
\end{eqnarray*}
Passing to a subsequence, we may assume that the sequence $\dfrac{\widehat{u}_k}{\|\widehat{u}_k\|}  $ converges to some ${u}' \ne 0.$ Note that
$\dfrac{u_k}{\|\widehat{u}_k\|} \to 0$ and $\dfrac{v_{k}}{\|\widehat{u}_k\|} \to 0.$ In view of Proposition~\ref{Prop37}, $u' \in D^*F_2(\overline{z},\overline{y}) (0)\cap \mathrm{ker} D^*F_1(\infty , \overline{z}),$ which contradicts the condition~(b). 
\end{proof}

\subsection{Subdifferentials at infinity}

We define next the {\em basic and singular limiting subdifferentials at infinity} of extended-real-valued functions geometrically via the normal cones at infinity to the epigraphs of the functions in question. Namely, let $f \colon \mathbb{R}^n \to \overline{\mathbb{R}}$ be a lower semi-continuous function and $\overline{y} \in J(f).$ Recall that
\begin{eqnarray*}
N_{\mathrm{epi} f}(\infty, \overline{y}) &:=& \Limsup_{(x,r)\stackrel{\mathrm{epi} f}\longrightarrow(\infty, \overline{y})}\widehat{N}_{\mathrm{epi} f}(x,r).
\end{eqnarray*}

\begin{definition}{\rm 
The {\em limiting and singular subdifferentials} of $f$ at $(\infty,\overline{y})$ are defined by 
\begin{eqnarray*}
\partial f(\infty,\overline{y}) &:=& \{u \in \mathbb{R}^n \ | \ (u, -1) \in N_{\mathrm{epi} f}(\infty,\overline{y})\},\\
\partial^{\infty} f(\infty,\overline{y}) &:=& \{u \in \mathbb{R}^n \ | \ (u, 0) \in N_{\mathrm{epi} f}(\infty,\overline{y})\}. 
\end{eqnarray*}
}\end{definition}

The next result provides a simple formula for the normal cones $N_{\mathrm{epi} f}(\infty, \overline{y}),$ and hence, for the subdifferentials $\partial f(\infty,\overline{y})$ and $\partial^{\infty} f(\infty,\overline{y}).$

\begin{proposition}
We have
\begin{eqnarray*}
N_{\mathrm{epi} f}(\infty, \overline{y}) &=& \Limsup_{(x,f(x))\longrightarrow(\infty, \overline{y})}\widehat{N}_{\mathrm{epi} f}(x,f(x)).
\end{eqnarray*}
\end{proposition}
\begin{proof}
The inclusion $\supset$ is clear from the definition of $N_{\mathrm{epi} f}(\infty, \overline{y}).$
For the inclusion $\subset,$ we see that for all $x \in \mathrm{dom} f$ and $r \geq f(x),$
\begin{eqnarray*}
T_{\textrm{epi} f}(x,f(x)) & \subset & T_{\textrm{epi} f}(x, r),
\end{eqnarray*}
which in combination with Lemma~\ref{Lemma2.3} yields
\begin{eqnarray*}
\widehat{N}_{\textrm{epi} f}(x,r)  & \subset  & \widehat{N}_{\textrm{epi} f}(x, f(x)).
\end{eqnarray*}
Therefore
\begin{eqnarray*}
\Limsup_{(x,r)\stackrel{\mathrm{epi} f}\longrightarrow(\infty, \overline{y})}\widehat{N}_{\mathrm{epi} f}(x,r) &\subset& \Limsup_{(x,f(x))\longrightarrow(\infty, \overline{y})}\widehat{N}_{\mathrm{epi} f}(x,f(x)),
\end{eqnarray*}
as required.
\end{proof}

The alternative formulas for the subdifferentials at infinity are given below.

\begin{proposition}\label{Prop312}
 We have
\begin{eqnarray*}
\partial f (\infty, \overline{y}) &=& \mathop{\mathop {\Limsup}\limits_{(x, f(x)) \to (\infty, \overline{y})}} { \partial f(x)}, \\
\partial^{\infty} f (\infty, \overline{y})  &=& \mathop{\mathop {\Limsup}\limits_{(x, f(x)) \to  (\infty, \overline{y})}}
\limits_{r  \searrow 0} {r \partial f (x)}.
\end{eqnarray*}
\end{proposition}
\begin{proof}
The proof is omitted because it is similar to \cite[Proposition~4.4]{PHAMTS2023-5}.
\end{proof}

\begin{example}{\rm
(i) Consider the smooth function $f \colon \mathbb{R} \to \mathbb{R}$,  $x\mapsto \sin x$. We have $J(f)=[-1, 1]$ and $\partial f(x)=\{\cos x\}$ for all $x\in \mathbb{R}$. By virtue of Proposition~\ref{Prop312}, $\partial f(\infty,0)=\{-1, 1\},$ which is not a singleton set, while $\partial^{\infty} f(\infty, 0) = \{0\}.$

(ii) Consider the smooth function $f \colon \mathbb{R} \to \mathbb{R}$,  $x\mapsto e^{x}$. We have $J(f)=\{0\}$ and $\partial f(x)=\{e^x\}$ for all $x\in \mathbb{R}$.  In view of Proposition~\ref{Prop312}, $\partial f(\infty,0)=\{0\}$ and $\partial^{\infty} f(\infty, 0) = [0, +\infty).$
}\end{example}

We consider the {\em epigraphical} set-valued mapping $E_{f} \colon \mathbb{R}^n \rightrightarrows \mathbb{R}$ associated with $f$ by
\begin{eqnarray*}
E_{f}(x) &:=& \{ r \in \mathbb{R}\mid  r\ge f(x) \}. 
\end{eqnarray*}
It is clear that $\textrm{gph}E_{f}=\textrm{epi}f$ and $\textrm{dom}E_{f}=\textrm{dom}f.$ Moreover we have
\begin{eqnarray*}
\partial f(\infty,\overline{y}) &=& D^*E_{f}(\infty,\overline{y})(1) \quad \textrm{ and } \quad  \partial^{\infty} f(\infty,\overline{y}) \ = \ D^*E_{f}(\infty,\overline{y})(0).
\end{eqnarray*}
The next result shows that we can replace $E_f$ in the coderivative representation of $\partial f(\infty,\overline{y}) $ by the function $f$ itself, having also a useful relationship between $\partial^{\infty} f(\infty,\overline{y}) $ 
and $D^* f (\infty,\overline{y})(0).$

\begin{proposition} 
Let $f$ be lower semi-continuous and $\overline{y} \in J(f).$ Then we have the following statements.
\begin{enumerate}[{\rm (i)}]
\item $\partial f (\infty, \overline{y})=D^* f (\infty,\overline{y})(1).$
\item $\partial^{\infty} f (\infty, \overline{y}) \subset D^* f (\infty,\overline{y})(0).$
\end{enumerate}
\end{proposition}

\begin{proof}
(i) Take any $u \in \partial f (\infty, \overline{y}).$ By Proposition~\ref{Prop312}, there are sequences $x_k \in \mathbb{R}^n$ and $u_k \in \partial f (x_k)$ such that $x_k \to \infty$, $f(x_k) \to \overline{y}$ and $u_k \to u.$ In view of \cite[Theorem~1.23]{Mordukhovich2018}, $u_k \in D^* f(x_k)(1)$. Letting $k \to \infty$ and employing Proposition~\ref{Prop37}, we obtain $u \in D^* f (\infty,\overline{y})(1).$

Conversely, let $u \in D^* f (\infty,\overline{y})(1)$. By Proposition \ref{Prop37}, we find sequences $x_k, u_k \in \mathbb{R}^n$ and $v_k \in \mathbb{R}$ with $u_k \in D^* f (x_k)(v_k)$ such that $x_k \to \infty, u_k \to u,$ $f (x_k) \to \overline{y}$ and $v_k \to 1.$
Then for all $k$ sufficiently large, we have $v_k \ne 0$ and so $\dfrac{u_k}{v_k} \in D^* f (x_k)(1),$ which, together with \cite[Theorem~1.23]{Mordukhovich2018}, yields $\dfrac{u_k}{v_k} \in \partial f (x_k).$ Letting $k \to \infty,$ we obtain $u \in \partial f (\infty, \overline{y}).$

(ii) Pick any $u \in \partial^{\infty} f (\infty, \overline{y}).$ By Proposition~\ref{Prop312}, there are sequences $x_k, u_k \in \mathbb{R}^n,$ with
$u_k \in \partial f (x_k),$ and $r_k \in \mathbb{R}$ such that $x_k \to \infty,$ $f(x_k) \to \overline{y},$ $r_k \searrow 0$ and $r_k u_k \to u$. By \cite[Theorem~1.23]{Mordukhovich2018}, we have $u_k \in D^* f(x_k)(1)$ and hence $r_k u_k \in D^* f(x_k)(r_k )$ as $r_k >0$.  Letting $k \to \infty$ and employing Proposition \ref{Prop37}, we obtain $u \in D^* f (\infty,\overline{y})(0).$
\end{proof}

\begin{remark}{\rm
The inclusion (ii) in the above proposition can be strict. A simple example on $\mathbb{R}^2$ is furnished by
 \begin{eqnarray*}
f(x_1, x_2) &:=& 
\begin{cases}
0 & \textrm{ if } x_1 < 0, \\
-\sqrt[3]{x_1} & \textrm{ if } x_1 \ge 0,
\end{cases}
\end{eqnarray*}
which has $J(f) = (-\infty, 0].$ For this function at $\overline{y} : = 0,$ we have
$\partial^{\infty} f (\infty, \overline{y}) = (-\infty, 0] \times \{0\},$ while $D^* f (\infty,\overline{y})(0) = (-\infty, +\infty) \times \{0\}.$
}\end{remark}

\section{Well-posedness and Mordukhovich's criterion at infinity} \label{Section4}

In this section, in the setting at infinity, basic well-posedness properties of set-values mappings will be characterized completely in terms of coderivatives.
Let us start with the following result in which the property in (ii) is openness around $(\infty, \overline{y})$ with linear rate $\mu,$ the property in (iii) is the metric regularity of $F$ around $(\infty, \overline{y})$ with constant $\ell,$ whereas the property in (iv) is the Lipschitz continuity of $F^{-1}$ around $(\overline{y}, \infty)$ with constant $\ell;$ for the related works, see
\cite{Aubin1984, Borwein1988, Mordukhovich1993, Penot1989, Rockafellar1985}

\begin{theorem}\label{DL41}
Let $F\colon \mathbb{R}^n\rightrightarrows \mathbb{R}^m$ be a set-valued mapping of closed graph, and let $\overline{y} \in J(F).$ Then the following conditions are equivalent:
\begin{enumerate}[\rm(i)]
\item {\rm (coderivative nonsingularity):} $\mathrm{ker}\,D^*F (\infty, \overline{y}) = \{0\}.$

\item {\rm (linear openness at infinity):} There exist constants $R > 0, \mu > 0, \epsilon > 0$ and a neighborhood $V$ of $\overline{y}$ in $\mathbb{R}^m$ such that
\begin{eqnarray*}
\big(F(x) + \mu r \mathbb{B} \big) \cap V &\subset& F(x+r\mathbb{B}) \quad \textrm{ whenever } \quad x \in \mathbb{R}^n \setminus  \mathbb{B}_R \ \textrm{ and } \  r \in (0, \epsilon). 
\end{eqnarray*}

\item {\rm (metric regularity at infinity):} There are constants $R > 0, \ell> 0, \gamma > 0,$ and a neighborhood $V$ of $\overline{y}$ in $\mathbb{R}^m$ such that
\begin{eqnarray*}
\mathrm{dist}(x, F^{-1}(y)) & \leq & \ell\, \mathrm{dist}(y, F(x))
\end{eqnarray*} 
for all $x \in \mathbb{R}^n \setminus \mathbb{B}_{R}$ and $y \in V$ with $\mathrm{dist}(y, F(x))  < \gamma.$
 
\item {\rm (inverse Lipschitz-like property at infinity):} There are constants $R > 0,  \ell > 0,$ and a neighborhood $V$ of $\overline{y}$ in $\mathbb{R}^m$ such that
\begin{eqnarray*}
F^{-1}(y') \cap \big( \mathbb{R}^n \setminus \mathbb{B}_{R}\big) & \subset & F^{-1}(y) + \ell\|y' - y\|\mathbb{B} \quad \textrm{for all} \quad y', y \in V. 
\end{eqnarray*} 
\end{enumerate}
\end{theorem}

\begin{proof} 
(i) $\Rightarrow$ (ii). Assume that $\mathrm{ker}\, D^*F (\infty, \overline{y}) = \{0\}.$ Since $D^*F (\infty, \overline{y}) $ is a closed positively homogeneous mapping, we have
\begin{eqnarray*}
\mu_* &:=& \inf\,\{\|u\| \mid u\in D^*F (\infty, \overline{y})(v) \textrm{ and }  \|v\| = 1 \} \ > \ 0.
\end{eqnarray*} 
Take any $\mu \in (0, \mu_*).$ We will show that there exist constants $R > 0, \epsilon > 0$ and a neighborhood $V$ of $\overline{y}$ in $\mathbb{R}^m$ such that
\begin{eqnarray*}
\big(F(x) + \mu r \mathbb{B} \big) \cap V &\subset& F(x + r \mathbb{B}) \quad \textrm{ whenever } \quad x \in \mathbb{R}^n \setminus  \mathbb{B}_R \ \textrm{ and } \ r \in (0, \epsilon). 
\end{eqnarray*}
By contradiction, there are sequences $(x_k, y_k) \to (\infty, \overline{y})$ and $r_k \searrow 0$ as well as $z_k\in\mathbb{R}^m$ satisfying the following conditions
\begin{eqnarray*}
y_k\in F(x_k), \quad \|z_k-y_k\|\leq \mu r_k, \quad z_k \notin F(x_k + r_k \mathbb{B}).
\end{eqnarray*}
Fix $k$ and define the function $\varphi \colon \mathbb{R}^n\times\mathbb{R}^m \to \overline{\mathbb{R}}$ by
\begin{eqnarray*}
\varphi (x, y) & := & \|y-z_k\|+\delta_{\mathrm{gph} F}(x, y),
\end{eqnarray*}
which is lower semi-continuous as the graph of $F$ is closed. Clearly $\inf_{(x, y) \in \mathbb{R}^n\times\mathbb{R}^m} \varphi (x, y) = 0$ and $\varphi (x_k, y_k) \to 0$ as $k \to \infty.$
Applying the Ekeland variational principle (see Theorem~\ref{Theorem210}) to $\varphi$ with $\epsilon := \mu r_k, \lambda := r_k,$ the initial point $(x_k, y_k),$ and the metric space 
$(\mathbb{R}^n\times\mathbb{R}^m, d),$ where 
\begin{eqnarray*}
d((x, y), (x', y')) &:=& \|x - x'\| + \sqrt{r_k} \|y - y'\| \quad \textrm{ for } \quad (x, y), (x', y') \in \mathbb{R}^n\times\mathbb{R}^m,
\end{eqnarray*}
we find a pair $(\overline{x}_k, \overline{y}_k) \in \mathrm{gph}\, F$ such that 
\begin{eqnarray*}
\|\overline{x}_k - x_k\| + \sqrt{r_k} \|\overline y_k - y_k\| & \leq & r_k
\end{eqnarray*}
and $(\overline{x}_k, \overline y_k)$ is a global minimizer of the lower semi-continuous function 
\begin{eqnarray*}
\psi \colon \mathbb{R}^n \times \mathbb{R}^m \to\overline{\mathbb{R}}, && (x, y) \mapsto \|y - z_k\| + {\mu} \big(\|x - \overline{x}_k\| + \sqrt{r_k} \|y - \overline{y}_k\| \big) + \delta_{\mathrm{gph} F}(x, y).
\end{eqnarray*}
By Fermat's rule (see Lemma~\ref{Lemma2.8}), then $(0, 0) \in \partial \psi (\overline{x}_k, \overline{y}_k).$ Note that $\overline{y}_k \ne z_k$ as $\overline{y}_k \in F(\overline{x}_k ) \subset F(x_k + r_k \mathbb{B})$ and $z_k \notin F(x_k + r_k \mathbb{B}).$ In view of Lemmas~\ref{Lemma2.6}, \ref{Lemma2.7} and \ref{Lemma2.9}, we have
\begin{eqnarray*}
(0, 0)  &\in& \{0\} \times \frac{\overline{y}_k - z_k}{\|\overline{y}_k - z_k\|}  + {\mu} \big( \mathbb{B} \times \sqrt{r_k} \mathbb{B} \big) + {N}_{\mathrm{gph} F}(\overline{x}_k, \overline{y}_k).
\end{eqnarray*}
Hence, there is a vector $(u_k, v_k) \in \mathbb{R}^n \times \mathbb{R}^m$ satisfying
\begin{eqnarray*}
\|u_k\| \leq {\mu},  \quad  1 - \mu \sqrt{r_k}  \le \|v_k\| \le 1 + \mu \sqrt{r_k} \quad \textrm{ and } \quad  u_k \in D^*F(\overline{x}_k, \overline{y}_k)(v_k).
\end{eqnarray*}
Passing to subsequences if necessary, we may assume that $(u_k, v_k)\to (u, v).$ Certainly $\|u\|\leq \mu$ and $\|v\| = 1.$ Observe that $(\overline{x}_k,  \overline{y}_k) \to (\infty, \overline{y}).$ Hence $u \in D^* F(\infty, \overline y)(v),$ which contradicts the choice of $\mu.$

(ii) $\Rightarrow$ (iii). Let $\ell := \frac{1}{\mu^{-1}}$ and $0 < \gamma < \mu \epsilon.$ Take any $x \in \mathbb{R}^n \setminus \mathbb{B}_R$ and $y\in V$ 
with $\mathrm{dist}(y, F(x))  < \gamma$ and put $r := \ell  \mathrm{dist}(y, F(x)) < \epsilon.$ Without loss of generality we can assume that $r > 0$ because otherwise the conclusion is trivial. Since the graph of $F$ is closed, there is $y' \in F(x)$ such that 
\begin{eqnarray*}
\|y - y'\| &=& \mathrm{dist}(y, F(x)) \ = \ \frac{r}{\ell}  \ = \ \mu r.
\end{eqnarray*} 
Hence
\begin{eqnarray*}
y &\in & y' + \mu r \mathbb{B} \ \subset \ F(x) + \mu r \mathbb{B}.
\end{eqnarray*} 
The condition (ii) implies that $y \in F(x + r \mathbb{B}).$ Therefore
\begin{eqnarray*}
\mathrm{dist}(x, F^{-1}(y)) &\leq&  r \ = \ \ell \mathrm{dist}(y, F(x)),
\end{eqnarray*}
as required.

(iii) $\Rightarrow$ (iv). Choose a positive constant $\epsilon < \gamma$ such that $\mathbb{B}_{\epsilon}(\overline{y}) \subset V.$  Let $y, y' \in \mathbb{B}_{\epsilon/2}(\overline{y})$ and take any $x \in F^{-1}(y') \cap \big( \mathbb{R}^n \setminus \mathbb{B}_{R}\big).$ We have $y \in V$ and 
\begin{eqnarray*}
\mathrm{dist}(y, F(x)) &\leq& \|y - y'\| \ \le \ \|y - \overline{y}\| + \|y' - \overline{y}\| \ \le \ \epsilon \ < \ \gamma.
\end{eqnarray*} 
The condition (iii) implies that
\begin{eqnarray*}
\mathrm{dist}(x, F^{-1}(y)) &\leq& \ell \, \mathrm{dist}(y, F(x)) \ \leq \ \ell \, \|y' - y\|,
\end{eqnarray*} 
which yields the desired inclusion.

(iv) $\Rightarrow$ (i). Assume that there are constants $R > 0,  \ell > 0$ and $\epsilon > 0$ such that
\begin{eqnarray}\label{3PT3}
F^{-1}(y') \cap \big( \mathbb{R}^n \setminus \mathbb{B}_{R}\big) & \subset & F^{-1}(y) + \ell\|y - y'\|\mathbb{B} \quad \textrm{for all} \quad y, y' \in V := \mathbb{B}_{\epsilon}(\overline{y}).
\end{eqnarray}
We will show that $\mathrm{ker}\, D^*F (\infty, \overline{y}) = \{0\}.$
To see this, let $v \in \mathbb{R}^m$ be such that $0 \in D^*F(\infty, \overline{y})(v).$ By Proposition~\ref{Prop37}, there exist sequences $(x_k, y_k) \in \mathrm{gph} F$ and $(u_k, -v_k)\in \widehat{N}_{\mathrm{gph} F}(x_k, y_k)$ such that $(x_k, y_k) \to (\infty, \overline{y})$ and $(u_k, v_k) \to (0, v)$ as $k \to \infty.$ In particular, $x_k \in \mathbb{R}^n \setminus \mathbb{B}_R$ and $y_k \in \overline{y} + \dfrac{\epsilon}{2}\mathbb{B}$ for all $k$ sufficiently large. If $v_k = 0$ for infinitely many $k,$ then $v = 0$ and there is nothing to prove. So, suppose that $v_k \ne 0$ for all $k.$

By the definition of  $\widehat{N}_{\mathrm{gph} F}(x_k, y_k),$ for each $k > 0$ we can find $\delta_k > 0$ such that 
\begin{eqnarray}\label{3PT4}
\langle u_k, x - x_k\rangle - \langle v_k, y - y_k\rangle & \leq & \frac{1}{k} \big(\|x - x_k\| + \|y - y_k\|\big)
\end{eqnarray}   
for all $(x, y) \in \mathrm{gph} F\cap \big (\mathbb{B}_{\delta_k}(x_k) \times \mathbb{B}_{\delta_k}(y_k)\big).$ 

Fix $k$ sufficiently large. Let $y' := y_k$ and $y := y_k - \frac{1}{\ell} r \frac{v_k}{\|v_k\|}$ with $0 < r < \min\{\delta_k, \dfrac{\ell \epsilon}{2} \}.$  Certainly $x_k \in F^{-1}(y') \setminus \mathbb{B}_R$ and $y, y' \in V.$ This, together with  the inclusion~\eqref{3PT3}, ensures the existence of $x \in F^{-1}(y)$ such that 
\begin{eqnarray*}
\|x - x_k\| & \le & \ell \|y - y'\| \ \le \ r.
\end{eqnarray*}
Therefore, by the inequality~\eqref{3PT4}, we get easily that
\begin{eqnarray*}
\frac{1}{\ell} r \|v_k\| &\leq& r \|u_k\|  + \frac{1}{k}( r + \frac{1}{\ell} r) ,
\end{eqnarray*} 
which yields
\begin{eqnarray*}
\|v_k\| &\leq& \ell \|u_k\|  + \frac{1}{k}(\ell + 1).
\end{eqnarray*} 
Finally, letting $k\to \infty$ we obtain $v = 0$ as required.
\end{proof}

\begin{example}{\rm 
Consider the set-valued mapping $F \colon \mathbb{R} \rightrightarrows \mathbb{R}$ defined by
$$F(x) := \{0\}\cup [x, +\infty).$$
We have $J(F) = \{0\},$ $N_{\mathrm{gph} F}(\infty, 0) = \{0\}\times \mathbb{R}$ and
$\mathrm{ker}\,D^*F (\infty, 0) = \mathbb{R}.$ Hence $F$ does not satisfy the equivalent conditions (ii)-(iv) in Theorem~\ref{DL41}.
}\end{example}

A set-valued mapping $F \colon \mathbb{R}^n\rightrightarrows \mathbb{R}^m$ can be identified in a set with its 
{\em distance function}
\begin{equation*}
d_F \colon \mathbb{R}^n \times \mathbb{R}^m\to \overline{\mathbb{R}}, \quad (x, y) \mapsto \mathrm{dist} \big (y, F(x) \big) = \inf \{\|y-v\|\;|\; v\in F(x)\}.
\end{equation*}
Clearly $F$ is uniquely determined by $d_F,$ so every property of $F$ must correspond to a property of $d_F$ and vice versa. The following result reveals the property of $d_F$ that corresponds to Lipschitz-like continuity of $F$ and indicates obviously why that notion has a natural significance; see also \cite[Theorem~2.3]{Rockafellar1985}.

\begin{proposition}[distance characterization of the Lipschitz-like property at infinity] \label{Prop43}
Let $F \colon \mathbb{R}^n\rightrightarrows \mathbb{R}^m$ be a set-valued mapping whose graph is closed, and let $\overline{y} \in J(F).$  The following conditions are equivalent:
\begin{enumerate}[{\rm (i)}]
\item $F$ is Lipschitz-like around $(\infty, \overline{y}),$ i.e., there exist constants $R > 0, \ell > 0$ and a neighborhood $V$ of $\overline{y}$ in $\mathbb{R}^m$ satisfying the inclusion
\begin{eqnarray}\label{3PT5}
F(x) \cap V &\subset& F(x') + \ell\|x - x'\| \mathbb{B} \quad \text{ for all } \quad x, x' \in F^{-1}(V) \setminus  \mathbb{B}_R. 
\end{eqnarray}

\item There exist constants $R > 0, \ell > 0$ and a neighborhood $V$ of $\overline{y}$ in $\mathbb{R}^m$ such that
\begin{eqnarray*}
\left | d_F(x, y) - d_F(x', y') \right| &\le& \ell \big (\|x - x'\| + \|y - y'\| \big)
\end{eqnarray*}
for all $x, x' \in F^{-1}(V) \setminus  \mathbb{B}_R$ and $y, y' \in V.$
\end{enumerate}
\end{proposition}

\begin{proof}
Observe first that (ii) is equivalent to the existence of constants $R > 0, \ell > 0$ and a neighborhood $V$ of $\overline{y}$ in $\mathbb{R}^m$ such that
for all $x, x' \in F^{-1}(V) \setminus  \mathbb{B}_R$ and $y \in V$ we have 
\begin{eqnarray} \label{3PT6}
\mathrm{dist}\big(y, F(x') \big)  &\le& \mathrm{dist}\big(y, F(x)\big)  + \ell \|x - x'\|
\end{eqnarray}
This is true because the function $d_F(x, y)$ is Lipschitz in $y$ with modulus $1$ for all $x \in F^{-1}(V).$

(ii) $\Rightarrow$ (i). If the inequality~\eqref{3PT6} holds for constants $R > 0, \ell > 0$ and some neighborhood $V$ of $\overline{y},$ then for every $x, x' \in F^{-1}(V) \setminus  \mathbb{B}_R$ and $y \in F(x) \cap V$ we have
\begin{eqnarray*}
\mathrm{dist}\big(y, F(x')\big) &\le& \mathrm{dist}\big(y, F(x)\big) + \ell \|x - x'\| \ = \ \ell \|x - x'\| ,
\end{eqnarray*}
because $\mathrm{dist}\big(y, F(x)\big)  = 0.$ Hence $y \in F(x') + \ell\|x - x'\| \mathbb{B},$ as required in \eqref{3PT5} for Lipschitz-like continuity.

(i) $\Rightarrow$ (ii). Assume that \eqref{3PT5}  holds for constants $R > 0, \ell > 0$ and some neighborhood $V$ of $\overline{y}.$ Then for all $x, x' \in F^{-1}(V) \setminus  \mathbb{B}_R$ and $y \in V$ we have
\begin{eqnarray*}
\mathrm{dist}\big(y, F(x') + \ell\|x - x'\| \mathbb{B} \big)  &\le& \mathrm{dist}\big(y, F(x) \cap V \big).
\end{eqnarray*}
Observe that
\begin{eqnarray*}
\mathrm{dist}\big(y, F(x')\big) - \rho &\le& \mathrm{dist}\big(y, F(x) + \rho \mathbb{B}\big)
\end{eqnarray*}
for any $\rho \ge 0,$ so
\begin{eqnarray*}
\mathrm{dist}\big(y, F(x')\big) - \ell \|x - x'\| &\le& \mathrm{dist}\big(y, F(x) \cap V\big).
\end{eqnarray*}
Therefore, it suffices to show that for a neighborhood $V'$ of $\overline{y}$ with $V' \subset V$ we have for all $x \in F^{-1}(V')$ and $y \in V',$ 
\begin{eqnarray*}
\mathrm{dist}\big(y, F(x) \cap V\big) &=& \mathrm{dist}\big(y, F(x)\big).
\end{eqnarray*}
To see this, choose $\epsilon > 0$ small enough that $\overline{y} + \epsilon \mathbb{B} \subset V,$ and let $V' := \overline{y} + (\epsilon/3) \mathbb{B}.$ Then $V' \subset V,$ and for arbitrary $x \in F^{-1}(V')$ and $y \in V'$ we have
\begin{eqnarray*}
\mathrm{dist}\big(y, F(x) \big) &\le& \|y - \overline{y}\| + \mathrm{dist}\big(\overline{y}, F(x) \big) \ \le \ \frac{\epsilon}{3} + \frac{\epsilon}{3} \ = \ \frac{2\epsilon}{3}.
\end{eqnarray*}
Since $F(x)$ is closed, there exists $y' \in F(x)$ such that $\|y - y'\| = \mathrm{dist}\big({y}, F(x) \big).$ Hence
\begin{eqnarray*}
\|y' - \overline{y} \| &\le& \|y' - y \| + \|y - \overline{y}\| \ \le \ \frac{2\epsilon}{3} + \frac{\epsilon}{3} \ = \ \epsilon,
\end{eqnarray*}
which yields $y' \in V,$ and so $\mathrm{dist}\big(y, F(x) \cap V\big) = \mathrm{dist}\big(y, F(x)\big),$ as required.
\end{proof}

It is well known from Mordukhovich's criterion (see \cite[Theorem~3.3(iii)]{Mordukhovich2018}) that 
for a closed graph set-valued mapping $F\colon \mathbb{R}^n\rightrightarrows \mathbb{R}^m$ and a point $(\overline{x}, \overline{y}) \in \mathrm{gph}F$ we have $D^*F(\overline{x}, \overline{y})(0) = \{0\}$ if and only if $F$ is Lipschitz-like around $(\overline{x}, \overline{y}),$ which means that there are a constant $\ell > 0$ and neighborhoods $U$ of  $\overline{x}$ and $V$ of $\overline{y}$ such that
\begin{eqnarray*}
F(x) \cap V &\subset& F(x') + \ell\|x - x'\| \mathbb{B} \quad \text{ for all } \quad x, x' \in U. 
\end{eqnarray*}
On the other hand, it will be demonstrated in Theorem~\ref{Thm45} and Example~\ref{VD44}(i) below that if $D^*F(\infty, \overline{y})(0) = \{0\}$ then $F$ is Lipschitz-like around $(\infty, \overline{y}),$ but not conversely.

\begin{example}\label{VD44} {\rm
(i) Consider the set-valued mapping $F \colon \mathbb{R} \rightrightarrows \mathbb{R}$ defined by
\begin{eqnarray*}
F(x) &:=& 
\begin{cases}
\{\frac{1}{x + 1} \} & \textrm{ if } x \in \mathbb{N},\\
\emptyset & \textrm{otherwise,}
\end{cases}
\end{eqnarray*}
which has $J(F) = \{0\}.$ By definition, it is easy to check that $D^*F(\infty, 0)(0) = \mathbb{R},$ while $F$ is Lipschitz-like around $(\infty, 0)$ with modulus $1.$ Note that for any open neighborhood $V$ of $0$ in $\mathbb{R},$ the inverse image $F^{-1}(V)$ is not open.

(ii) Consider the set-valued mapping $F \colon \mathbb{R}^2 \rightrightarrows \mathbb{R}$ defined by $F(x_1, x_2) := \{x_1 - x_2^2\}.$
We have $0 \in J(F)$ and $D^*F(\infty, 0)(0) = \{0\} \times \mathbb{R}.$ Moreover, $F$ is not Lipschitz-like around $(\infty, 0),$ while $F^{-1}(V)$ is open for any open subset $V$ of $\mathbb{R}.$
}\end{example}

The equivalence of (i) with  (ii) in the next result can be seen as a {\em version at infinity of Mordukhovich's criterion}.

\begin{theorem}\label{Thm45}
Let $F\colon \mathbb{R}^n\rightrightarrows \mathbb{R}^m$ be a set-valued mapping of closed graph, and let $\overline{y} \in J(F).$ 
The following conditions are equivalent:
\begin{enumerate}[{\rm (i)}]
\item $D^*F(\infty, \overline{y})(0)=\{0\}.$

\item There exist constants $R > 0, \ell > 0$ and an open neighborhood $V$ of $\overline{y}$ in $\mathbb{R}^m$ such that 
the set $F^{-1}(V) \setminus \mathbb{B}_R$ is open and one has
\begin{eqnarray*}
F(x) \cap V &\subset& F(x') + \ell\|x - x'\| \mathbb{B} \quad \text{ for all } \quad x, x' \in F^{-1}(V) \setminus  \mathbb{B}_R. 
\end{eqnarray*}

\item There exist constants $R > 0, \ell > 0$ and a neighborhood $V$ of $\overline{y}$ in $\mathbb{R}^m$ such that the function $d_F$ is Lipschitz continuous with modules $\ell$ on some neighborhood of each point of $\mathrm{gph} F \cap \big( (\mathbb{R}^n \setminus  \mathbb{B}_R) \times V \big).$

\item $\mathscr{F}^\infty = \{0\}.$
\end{enumerate}
When these equivalent conditions hold, one has moreover that
\begin{eqnarray*}
\|u\| &\leq& \ell \quad \textrm{ whenever } \quad u \in \mathscr{F},
\end{eqnarray*}
where we let 
\begin{eqnarray*}
\mathscr{F} &:=& \mathop{\mathop {\Limsup}\limits_{(x, y) \xrightarrow{\mathrm{gph} F}(\infty, \overline{y})}} { \partial d_F (x, y)}
\quad \textrm{ and } \quad 
\mathscr{F}^\infty \ := \  \mathop{\mathop {\Limsup}\limits_{(x, y) \xrightarrow{\mathrm{gph} F}(\infty, \overline{y})} }
\limits_{r  \searrow 0} {r \partial d_F (x, y)}.
\end{eqnarray*}
\end{theorem}

\begin{proof}
We will show that (i) $\Leftrightarrow$ (iii) $\Leftrightarrow$ (iv), (i) and (iii) $\Rightarrow$ (ii), and (ii) $\Rightarrow$ (iii).

(i) $\Rightarrow$ (iii). Let
\begin{eqnarray*}
\ell_* &:=& \sup \,\{\|u\| \mid u \in D^*F (\infty, \overline{y})(v) \textrm{ and } \|v\|  \le  1 \}.
\end{eqnarray*} 
Our assumption, together with the positive homogeneity of coderivative mappings, ensures that $\ell_* < +\infty.$ Take any $\ell > \ell_*.$ 
We will show that there exist a constant $R > 0$ and a neighborhood $V$ of $\overline{y}$ in $\mathbb{R}^m$ such that for all $(x, y) \in \mathrm{gph} F \cap \big( (\mathbb{R}^n \setminus \mathbb{B}_R) \times V \big)$ the following inequality holds
\begin{eqnarray*}
\|D^*F(x, {y})\| &\le& \ell.
\end{eqnarray*}
This claim, of course, together with \cite[Theorem~3.3]{Mordukhovich2018} and \cite[Theorem~2.3]{Rockafellar1985}, implies that the function $d_F$ is Lipschitz on some neighborhood of $(x, y)$ with modules $\ell,$ which proves~(iii).

Indeed, if the claim were not true, there would exist a sequence $(x_k, y_k) \in \mathrm{gph} F$ tending to $(\infty, \overline{y})$ such that 
\begin{eqnarray*}
\|D^*F(x_k, y_k)\| &>& \ell.
\end{eqnarray*}
By definition, we can find $u_k \in \mathbb{R}^n$ and $v_k \in \mathbb{R}^m$ with $u_k \in D^*F(x_k, y_k)(v_k)$ so that $\|u_k\| > \ell \|v_k\|.$
Clearly, $u_k \ne 0,$ $\frac{\|v_k\|}{\|u_k\|} < \frac{1}{\ell}$ and $\frac{u_k}{\|u_k\|} \in D^*F(x_k, y_k)(\frac{v_k}{\|u_k\|}).$ 
Let $u$ and $v$ be, respectively, cluster points of the (bounded) sequences $\frac{u_k}{\|u_k\|}$
and $\frac{v_k}{\|u_k\|}.$ Certainly $\|u\| = 1 \ge \ell \|v\|,$ $u \in D^*F(\infty, \overline{y})(v)$ (by Proposition~\ref{Prop37}). On the other hand, from the definition of $\ell_*$ we deduce that $\|u\| \le \ell_* \|v\|.$ Therefore $v \ne 0,$ and so $\ell_* \ge \ell,$ in contradiction to the choice of $\ell.$

(iii) $\Rightarrow$ (i). Take any $u \in D^*F(\infty, \overline{y})(0).$ By Proposition~\ref{Prop37}, there exist sequences $(x_k, y_k) \in \mathrm{gph}F$ and $(u_k, v_k) \in \mathbb{R}^n \times \mathbb{R}^m$ with $u_k \in D^*F(x_k, y_k)(v_k)$ such that
\begin{eqnarray*}
\|x_k\| \to \infty, \quad y_k \to \overline{y}, \quad u_k \to u, \quad \textrm{ and } \quad v_k \to 0.
\end{eqnarray*}
For all $k$ large enough, the function $d_F$ is Lipschitz continuous with modules $\ell$ on some neighborhood of $(x_k, y_k),$ which, together with \cite[Theorem~2.3]{Rockafellar1985} and \cite[Theorem~3.3(iii)]{Mordukhovich2018}, implies that $\|u_k\| \le \ell \|v_k\|.$ Letting $k \to \infty$ we get $\|u\| = 0,$ as required.

(i) and (iii) $\Rightarrow$ (ii). Let
\begin{eqnarray*}
\ell_* &:=& \sup \,\{\|u\| \mid u \in D^*F (\infty, \overline{y})(v) \textrm{ and }  \|v\|  \le  1 \}.
\end{eqnarray*} 
Our assumption, together with the positive homogeneity of coderivative mappings, ensures that $\ell_* < +\infty.$ Let $\ell > \ell_*.$ 

We first show that $F$ is Lipschitz-like around $(\infty, \overline{y})$ with modulus $\ell.$
By contradiction, assume that there are sequences $x_k \to \infty,$ $x_k' \to\infty$ and $y_k \in F(x_k)$ satisfying the following conditions
\begin{eqnarray*}
y_k \to \overline{y}, \quad \mathrm{dist}(\overline{y}, F(x_k')) \to 0, \quad 
\mathrm{dist}(y_k, F(x_k')) &> & \ell \|x_k - x_k'\|.
\end{eqnarray*}
Certainly $r_k := \ell \|x_k - x_k'\| > 0$ and $x_k' \not \in F^{-1}(y_k + r_k \mathbb{B}).$

Fix $k$ and define the function $\varphi \colon \mathbb{R}^n\times\mathbb{R}^m \to \overline{\mathbb{R}}$ by
\begin{eqnarray*}
\varphi (x, y) & := & \|x - x_k'\| + \delta_{\mathrm{gph} F}(x, y),
\end{eqnarray*}
which is lower semi-continuous as the graph of $F$ is closed. Clearly $\inf_{(x, y) \in \mathbb{R}^n\times\mathbb{R}^m} \varphi (x, y) = 0$ and $\varphi (x_k, y_k) > 0.$ Applying Theorem~\ref{Theorem210} to $\varphi$ with $\epsilon := \varphi (x_k, y_k) = \ell^{-1} r_k, \lambda := r_k,$ the initial point $(x_k, y_k),$ and the metric $d \colon \mathbb{R}^n\times\mathbb{R}^m\to\mathbb{R}, (x, y) \mapsto \sqrt{r_k}  \|x\| + \|y\|,$ we find a pair $(\overline{x}_k, \overline{y}_k) \in \mathrm{gph}\, F$ such that 
\begin{eqnarray*}
\sqrt{r_k}  \|\overline{x}_k - x_k\| + \|\overline y_k - y_k\| & \leq & r_k
\end{eqnarray*}
and $(\overline{x}_k, \overline y_k)$ is a global minimizer of the lower semi-continuous function 
\begin{eqnarray*}
\psi \colon \mathbb{R}^n \times \mathbb{R}^m \to\overline{\mathbb{R}}, && (x, y) \mapsto \|x - x_k'\| + {\ell^{-1}} \big(\sqrt{r_k} \|x - \overline{x}_k\| + \|y - \overline{y}_k\| \big) + \delta_{\mathrm{gph} F}(x, y).
\end{eqnarray*}
By Fermat's rule (see Lemma~\ref{Lemma2.8}), then $(0, 0) \in \partial \psi (\overline{x}_k, \overline{y}_k).$ Note that $\overline{x}_k \ne x_k'$ as $\overline{x}_k \in F^{-1}(\overline{y}_k ) \subset F^{-1}(y_k + r_k \mathbb{B})$ and $x_k' \notin F^{-1}(y_k + r_k \mathbb{B}).$ In view of Lemmas~\ref{Lemma2.6} and \ref{Lemma2.7}, we have
\begin{eqnarray*}
(0, 0) 
&\in & \frac{\overline{x}_k - x_k'}{\|\overline{x}_k - x_k'\|} \times \{0\}  + {\ell^{-1}} \big( \sqrt{r_k}  \mathbb{B} \times \mathbb{B} \big) + {N}_{\mathrm{gph} F}(\overline{x}_k, \overline{y}_k),
\end{eqnarray*}
or equivalently,
\begin{eqnarray*}
(0, 0) 
&\in & {\ell} \frac{\overline{x}_k - x_k'}{\|\overline{x}_k - x_k'\|}  \times \{0\} + \big( \sqrt{r_k}  \mathbb{B} \times \mathbb{B} \big) + {N}_{\mathrm{gph} F}(\overline{x}_k, \overline{y}_k).
\end{eqnarray*}
Hence, there is a vector $(u_k, v_k)\in \mathbb{R}^n\times \mathbb{R}^m$ satisfying
\begin{eqnarray*}
\ell - \sqrt{r_k}  \le \|u_k\| \le \ell  + \sqrt{r_k},  \quad \|v_k\| \leq 1,  \quad \textrm{ and } \quad  u_k \in D^*F(\overline{x}_k, \overline{y}_k)(v_k).
\end{eqnarray*}
Passing to subsequences if necessary, we may assume that $(u_k, v_k)\to (u, v).$ Certainly $\|u\|  = \ell$ and $\|v\| \le 1.$ By our construction,  $\overline{x}_k\to\infty$ and $\overline{y}_k \to \overline{y}.$ In view of Proposition~\ref{Prop37}, $u \in D^* F(\infty, \overline y)(v),$ which contradicts the choice of $\ell.$

Therefore, $F$ is Lipschitz-like around $(\infty, \overline{y})$ with modulus $\ell,$ i.e, there exist a constant $R > 0$ and an open neighborhood $V$ of $\overline{y}$ in $\mathbb{R}^m$ such that 
\begin{eqnarray*}
F(x) \cap V &\subset& F(x') + \ell\|x - x'\| \mathbb{B} \quad \text{ for all } \quad x, x' \in F^{-1}(V) \setminus  \mathbb{B}_R. 
\end{eqnarray*}
Moreover, since the condition~(iii) holds, by increasing $R$ and shrinking $V$ if necessary, we may assume that the function $d_F$ is Lipschitz continuous with modules $\ell$ on some neighborhood of each point of $\mathrm{gph} F \cap \big( (\mathbb{R}^n \setminus  \mathbb{B}_R) \times V \big).$ Hence, to complete the proof of this implication, it suffices to prove that $F^{-1}(V) \setminus \mathbb{B}_R$ is an open set. To see this, let $a \in F^{-1}(V) \setminus \mathbb{B}_R$ and we will show that $a$ is an interior point of $F^{-1}(V) \setminus \mathbb{B}_R.$ Indeed, by definition, there is $b \in F(a) \cap V,$ which yields $b + \epsilon \mathbb{B} \subset V$ for some $\epsilon > 0$ because $V$ is open. Choose a positive constant $\eta < \frac{\epsilon}{\ell}$ small enough that $a + \eta \mathbb{B} \subset \mathbb{R}^n \setminus \mathbb{B}_R$ and the function $x \mapsto d_F(x, b)$ is Lipschitz with modulus $\ell$ on $a + \eta \mathbb{B}.$ Take any $x \in  a + \eta \mathbb{B}.$ We have
\begin{eqnarray*}
\mathrm{dist}(b, F(x)) \ = \ d_F(x, b) &\le& d_F(a, b) + \ell\|x - a\| \ \le \ \ell \eta \ < \ \epsilon.
\end{eqnarray*}
Since the set $F(x)$ is closed, there is $y \in F(x)$ such that
\begin{eqnarray*}
\|b - y\| \ = \ \mathrm{dist}(b, F(x))  \ < \ \epsilon,
\end{eqnarray*}
which yields $y \in b + \epsilon \mathbb{B} \subset V,$ and so $x \in F^{-1}(V).$ 
This being true for arbitrary $x \in a + \eta \mathbb{B} \subset \mathbb{R}^n \setminus \mathbb{B}_R,$ we see that $a + \eta \mathbb{B} \subset F^{-1}(V) \setminus \mathbb{B}_R,$ which yields that $a$ is an interior point of $F^{-1}(V) \setminus \mathbb{B}_R,$ as required.

(ii) $\Rightarrow$ (iii). 
Take any $(a, b) \in \mathrm{gph} F \cap \big( (\mathbb{R}^n \setminus  \mathbb{B}_R) \times V \big).$ Then $b + \epsilon \mathbb{B} \subset V$
for some $\epsilon > 0.$ Moreover, since the set $F^{-1}(V) \setminus \mathbb{B}_R$ is open, there exists a constant $\eta > 0$ such that $a + \eta \mathbb{B} \subset F^{-1}(V) \setminus  \mathbb{B}_R.$ Our assumption, together with Proposition~\ref{Prop43}, implies that the function $d_F$ is Lipschitz with modulus $\ell$ on $(a + \eta \mathbb{B}) \times (b + \epsilon \mathbb{B}),$ and this is what was required.

(iii) $\Rightarrow$ (iv). 
Observe first that $d_F$ is nonnegative on $\mathbb{R}^n \times \mathbb{R}^m$ and is equal to zero on the graph of $F.$ By Fermat's rule (see Lemma~\ref{Lemma2.8}), then $0 \in \partial d_F (x, y)$ for all $(x, y) \in \mathrm{gph} F.$ By definition, $0 \in \mathscr{F}^\infty.$ 
Hence, it remains only to verify that $\mathscr{F}^\infty \subset \{0\}.$ To see this, take any $w \in \mathscr{F}^\infty.$ By definition,  there exist sequences $(x_k, y_k) \xrightarrow{\mathrm{gph} F} (\infty, \overline{y}),$ $r_k \searrow 0$ and $w_k \in \partial d_F (x_k, y_k)$ such that $r_k w_k \to w.$ Then for all $k$ sufficiently large we have $(x_k, y_k) \in \mathrm{gph} F \cap \big( (\mathbb{R}^n \setminus  \mathbb{B}_R) \times V \big).$ Since $d_F$ is Lipschitz continuous with modules $\ell$ around $(x_k, y_k)$ for such $k,$ it follows from \cite[Theorem~1.22]{Mordukhovich2018} that $\|w_k\| \leq \ell.$ Therefore $ r_k w_k \to 0$ and so $w = 0,$ as required.

(iv) $\Rightarrow$ (iii). By \cite[Theorem~1.22]{Mordukhovich2018} again, it suffices to verify that there exist constants $R > 0, \ell > 0$ and a neighborhood $V$ of $\overline{y}$ such that for all $(x, y) \in \mathrm{gph} F \cap \big( (\mathbb{R}^n \setminus  \mathbb{B}_R) \times V \big)$ we have
\begin{eqnarray*}
\|w\| &\leq& \ell \quad \textrm{ whenever } \quad w \in \partial d_F (x, y).
\end{eqnarray*}
Indeed, if this claim is not true, we can find sequences $(x_k, y_k) \xrightarrow{\mathrm{gph} F} (\infty, \overline{y}) $  and $w_k \in \partial d_F(x_k, y_k)$ such that $\|w_k\| >  k$ for all $k.$ Passing to a subsequence if necessary we can assume that the sequence $\frac{1}{\|w_k\|}w_k$ converges to a point $w \in \mathbb{R}^n \times \mathbb{R}^m.$ Certainly $w \in \mathscr{F}^\infty$ and $\|w\| = 1,$ a contradiction.

Finally, the last statement of the theorem follows directly from the above arguments.
\end{proof}

\begin{example}{\rm 
Let $F \colon \mathbb{R} \rightrightarrows \mathbb{R}$ be the set-valued mapping defined by
\begin{eqnarray*}
F(x) &:=&
\begin{cases}
(-\infty, 0] & \textrm{ if } x \leq 0,\\
\{x^2\} &\textrm{ otherwise}.
\end{cases}
\end{eqnarray*}
Let $\overline{y} := 0 \in J(F) = (-\infty, 0].$ An easy computation shows that $N_{\mathrm{gph}\,F}(\infty, \overline{y}) = \{0\}\times [0, \infty),$ and so  $D^*F(\infty, \overline{y})(0) = \{0\}.$ Hence the equivalent conditions (ii)-(iv) in Theorem~\ref{Thm45} are satisfied.
}\end{example}

\section{Fermat's rule at infinity} \label{Section5}

In this section we establish {\em necessary optimality conditions at infinity} for weak efficient values for set-valued optimization problems. 

Let $K \subsetneq \mathbb{R}^m$ be a pointed closed convex cone with nonempty interior which defines a partial order $\leqslant_{K}$ on $\mathbb{R}^m$ by the relation:
\begin{eqnarray*}
y_1 \leqslant_{K} y_2 \quad \textrm{ if and only if } \quad y_2 - y_1 \in K.
\end{eqnarray*}
Recall that the {\em positive polar} of $K$ is defined to be the set
\begin{eqnarray*}
K^{+} &:=& \{ c^{*} \in \mathbb{R}^m \mid \langle c^{*}, y \rangle \geq 0 \quad \textrm{ for all } \quad y \in K \}.
\end{eqnarray*}
By definition, $K^{+}$ is a closed convex cone. 

Let $F \colon \mathbb{R}^n \rightrightarrows \mathbb{R}^m$ be a set-valued mapping of closed graph and $\Omega$ be an unbounded closed subset of $\mathbb{R}^n.$ Consider the set-valued optimization problem with explicit geometry constraints
\begin{equation} \label{Problem}
\mathrm{minimize}_{K} \ F (x) \quad \textrm{ subject to } \quad x \in \Omega. \tag{P}
\end{equation}

\begin{definition}{\rm
We say that $\overline{y} \in \cl F(\Omega)$ is a {\em weak efficient value} for \eqref{Problem} if
\begin{eqnarray*}
\left(\overline{y} - {\rm int} K\right)\cap F(\Omega) &=& \emptyset. 
\end{eqnarray*}
}\end{definition}

Let $\overline{y} \in \cl F(\Omega)$ be a weak efficient value for \eqref{Problem}. It is well known in the theory of set-valued optimization (see, for example, \cite[Theorem~9.22 and Corollary~9.23]{Mordukhovich2018}) that if $\overline{y} \in F(\overline{x})$ 
for some $\overline{x} \in \Omega$ satisfying the constraint qualification
\begin{eqnarray*}
D^*F(\overline{x}, \overline{y})(0) \cap \big(-N_{\Omega}(\overline{x}) \big) &=& \{0\},
\end{eqnarray*}
then there exists a vector $c^* \in K^{+} \setminus \{0\}$ such that
\begin{eqnarray*}
0 &\in& D^*F(\overline{x}, \overline{y}) (c^*) + N_{\Omega}(\overline{x}).
\end{eqnarray*}
Here, we are interested in the case where $\overline{y} \not \in F(\Omega).$ To this end, recall from \cite[Definition~3.1]{PHAMTS2023-5} that the {\em normal cone to the set $\Omega$ at infinity} is defined by
\begin{eqnarray*}
N_{\Omega}(\infty) &:=& \Limsup_{x \xrightarrow{\Omega} \infty} \widehat{N}_{\Omega}(x).
\end{eqnarray*} 
It is not hard see that (see \cite[Proposition~3.5]{PHAMTS2023-5})
\begin{eqnarray*}
N_{\Omega}(\infty) &=& \Limsup_{x \xrightarrow{\Omega} \infty} N_{\Omega}(x).
\end{eqnarray*}
Now we are in a position to state our main result of this section concerning necessary optimality conditions for \eqref{Problem}.

\begin{theorem}[Fermat's rule at infinity] \label{Fermatatinfinity}
Let $\overline{y} \in \cl F(\Omega) \setminus F(\Omega)$ be a weak efficient value for \eqref{Problem} satisfying the constraint qualification at infinity
\begin{eqnarray*}
D^*F(\infty, \overline{y})(0) \cap \big(-N_{\Omega}(\infty)\big) &=& \{0\}.
\end{eqnarray*}
Then there exists a vector $c^* \in K^{+} \setminus \{0\} $ such that
\begin{eqnarray*}
0 &\in& D^*F(\infty, \overline{y}) (c^*) + N_{\Omega}(\infty).
\end{eqnarray*} 
\end{theorem}

In order to prove the theorem, we fix an element $e \in {\rm int} K$ and define the function $\varphi \colon \mathbb{R}^m \to \mathbb{R}$ by
\begin{eqnarray*}
\varphi(y) &:=& \inf \{ \lambda \in \mathbb{B} \mid \lambda e \in y + K\}.
\end{eqnarray*}

\begin{lemma} \label{Tammerbook}
The following properties hold:
\begin{enumerate}[{\rm (i)}]
\item $\varphi$ is well defined and globally Lipschitz on $\mathbb{R}^m.$
\item For any $y \in \mathbb{R}^m,$ we have
\begin{eqnarray*}
\partial \varphi (y) &=& \{ c^{*} \in K^+ \mid  \langle c^{*}, e \rangle =1, \langle c^{*},y\rangle = \varphi (y) \}.
\end{eqnarray*} 
\end{enumerate}
\end{lemma}

\begin{proof}
See \cite[Lemma~2.4]{Durea2009} and \cite[Section~5.2]{Khan2015}.
\end{proof}

\begin{proof}[Proof of Theorem~\ref{Fermatatinfinity}]
There are two cases to be considered.

\subsubsection*{Case 1: $\Omega$ is the whole space $\mathbb{R}^n$}
In this case, $N_{\Omega}(\infty) = \{0\},$ and so we must show that $0 \in D^*F(\infty , \overline{y}) (c^*)$ for some $c^* \in K^{+} \setminus \{0\}.$

Since $\overline{y} \in \cl F(\mathbb{R}^n)\setminus F(\mathbb{R}^n),$ there is a sequence $(x_k, y_k) \in \mathrm{gph} F$ such that $y_k \to \overline{y}.$ If the sequence $x_k$ has a cluster point $x^* \in \mathbb{R}^n,$ then $\overline{y} \in F(x^*)$ because the graph of $F$ is closed, which is a contradiction. Hence, $\lim_{k \to \infty} \|x_k\|=\infty.$

Define the Lipschitz function $\psi \colon \mathbb{R}^n \times \mathbb{R}^m \to \mathbb{R}$ by $\psi(x, y) := \varphi(y - \overline{y}).$ We first claim that $\psi$ is nonnegative on the graph of $F.$ Indeed, if this were not true, there would exist $({x}, {y}) \in \mathrm{gph} F $ such that $\psi(x, y) < 0.$ By definition, there exist ${\lambda} < 0$ and $y' \in K$ such that ${\lambda} e = {y} - \overline{y} + y'.$ It follows that
\begin{eqnarray*}
\overline{y} - {y} &=& -{\lambda} e + y' \ \in \ \mathrm{int}K + K \ \subset \ \mathrm{int} K,
\end{eqnarray*}
which contradicts our assumption that $\overline{y}$ is a weak efficient value for \eqref{Problem}.

Note that $\psi(x_k, y_k) \to 0$ as $k \to \infty.$ Therefore,
\begin{eqnarray*}
\inf_{(x,y)\in \mathbb{R}^n \times \mathbb{R}^m} \big (\psi + \delta_{\mathrm{gph} F} \big)(x, y) &=& 
\inf_{(x, y) \in {\mathrm{gph} F}} \psi(x, y) \ = \ 0.
\end{eqnarray*}
(Recall that $\delta_{\mathrm{gph} F}$ stands for the indicator function of the set $\mathrm{gph} F.$) By passing to subsequences, it suffices to consider the following two cases.

\subsubsection*{Case 1.1: $\psi(x_k, y_k) = 0$ for all $k$}
Then $(x_k,y_k)$ is a global minimizer of the optimization problem 
\begin{eqnarray*}
\min_{(x,y) \in \mathbb{R}^n \times \mathbb{R}^m} \big (\psi + \delta_{\mathrm{gph} F} \big)(x, y). 
\end{eqnarray*}
It follows from Lemma~\ref{Lemma2.8} that
\begin{eqnarray*}
(0, 0) &\in& \partial \big( \psi +\delta_{\mathrm{gph} F} \big)(x_k, y_k).
\end{eqnarray*}
Since the function $\psi$ is Lipschitz, the sum rule (see Lemma~\ref{Lemma2.9}) gives
\begin{eqnarray*}
(0, 0) 
&\in& \partial \psi(x_k, y_k) + \partial \delta_{\mathrm{gph} F}(x_k, y_k) \\
&=& \partial \psi(x_k, y_k) + N_{{\mathrm{gph} F}}\big((x_k, y_k) \big) 
\end{eqnarray*}
On the other hand, by definition, it is not hard to see that
\begin{eqnarray*}
\partial \psi(x_k, y_k) & \subset & \{0\} \times \partial \varphi(y_k - \overline{y}).
\end{eqnarray*}
Therefore, in view of Lemma~\ref{Tammerbook}, there is $c_k ^* \in \partial \varphi(y_k - \overline{y}) \subset K^{+}$ with $\langle c^*_k, e \rangle =1$ such that 
\begin{eqnarray*}
(0, -c^*_k) & \in & N_{{\mathrm{gph} F}}\big(({x}_k, {y}_k) \big). 
\end{eqnarray*}
Since the function $\varphi_{{e}}$ is globally Lipschitz, the sequence $c^*_k$ must be bounded and so it has a cluster point $c^*.$ 
Certainly, $c^* \in K^{+}$ and $\langle c^*, e \rangle =1,$ and so $c^* \ne 0.$ Moreover, we have $(0, -c^*) \in N_{{\mathrm{gph} F}} (\infty, \overline{y}),$ which means that $0 \in D^*F(\infty , \overline{y}) (c^*),$ as required.

\subsubsection*{Case 1.2: $\epsilon_k:= \psi(x_k, y_k) > 0$ for all $k$}
Employing the Ekeland variational principle (see Theorem~\ref{Theorem210}), we find $(\overline{x}_k, \overline{y}_k) \in \mathrm{gph} F$ such that 
\begin{eqnarray*}
\|(x_k,y_k) - (\overline{x}_k, \overline{y}_k)\| & \le & \sqrt{\epsilon_k},
\end{eqnarray*}
and for all $(x,y) \in \mathbb{R}^n \times \mathbb{R}^m$ we have
\begin{eqnarray*}
\psi(\overline{x}_k, \overline{y}_k) + \delta_{\mathrm{gph} F}(\overline{x}_k,\overline{y}_k) &\le& \psi(x, y) + \delta_{\mathrm{gph} F}(x,y)+\sqrt{\epsilon_k}\|(x,y)- (\overline{x}_k, \overline{y}_k)\|.
\end{eqnarray*}
The first condition implies that $(\overline{x}_k, \overline{y}_k) \to (\infty, \overline{y})$ as $k \to \infty,$ while the second condition implies that
$(\overline{x}_k,\overline{y}_k)$ is a global minimizer of the optimization problem
\begin{eqnarray*}
\min_{(x,y)\in \mathbb{R}^n \times \mathbb{R}^m} \big( \psi + \delta_{\mathrm{gph} F} + \sqrt{\epsilon_k}\|\cdot - (\overline{x}_k, \overline{y}_k)\| \big) (x, y). 
\end{eqnarray*}
By Fermat's rule (see Lemma~\ref{Lemma2.8}), we have
\begin{eqnarray*}
(0, 0) &\in& \partial \big( \psi +\delta_{\mathrm{gph} F}+ \sqrt{\epsilon_k}\|\cdot-(\overline{x}_k, \overline{y}_k)\|\big)(\overline{x}_k, \overline{y}_k).
\end{eqnarray*}
Since the functions $\psi$ and $\|\cdot-(\overline{x}_k, \overline{y}_k)\|$ are Lipschitz, the sum rule (see Lemma~\ref{Lemma2.9}) gives
\begin{eqnarray*}
(0, 0)  &\in& \partial \psi(\overline{x}_k, \overline{y}_k) + \partial \delta_{\mathrm{gph} F}(\overline{x}_k, \overline{y}_k)+ \sqrt{\epsilon_k}\partial \big (\|\cdot-(\overline{x}_k, \overline{y}_k)\| \big )(\overline{x}_k, \overline{y}_k) \\
&\subset& \{0\} \times \partial \varphi(\overline{y}_k - \overline{y}) + N_{{\mathrm{gph} F}}(\overline{x}_k, \overline{y}_k) + \sqrt{\epsilon_k}\partial \big (\|\cdot-(\overline{x}_k, \overline{y}_k)\| \big )(\overline{x}_k, \overline{y}_k) \\
&=& \{0\} \times \partial \varphi(\overline{y}_k - \overline{y}) + N_{{\mathrm{gph} F}}(\overline{x}_k, \overline{y}_k) + \sqrt{\epsilon_k} \mathbb{B}.
\end{eqnarray*}
This, together with Lemma~\ref{Tammerbook}, yields the existence of $(u_k,v_k) \in \mathbb{R}^n \times \mathbb{R}^m$ with $\|(u_k,v_k)\| \leq \sqrt{\epsilon_k}$ and $c_k ^* \in \partial \varphi(\overline{y}_k - \overline{y}) \subset K^{+}$ with $\langle c^*_k,e \rangle =1$ such that 
\begin{eqnarray*}
(u_k,v_k - c^*_k) & \in & N_{{\mathrm{gph} F}}(\overline{x}_k, \overline{y}_k). 
\end{eqnarray*}
As in the proof of Case~1.1, we can see that the sequence $c_k ^*$ has a cluster point $c^* \in K^{+} \setminus\{0\}.$ Since $(u_k,v_k) \to (0, 0)$ as
$k \to \infty,$ we get
$(0, -c^*) \in N_{{\mathrm{gph} F}} (\infty,\overline{y}),$ which means that $0 \in D^*F(\infty , \overline{y}) (c^*),$ as required.

\subsubsection*{Case 2: $\Omega$ is a proper subset of $\mathbb{R}^n$}

Define the set-valued mapping $\Delta_{\Omega} \colon \mathbb{R}^n \rightrightarrows \mathbb{R}^m$ by
\begin{eqnarray*}
\Delta_{\Omega} (x) &:=&
\begin{cases}
\{0\} &\textrm{ if } x \in \Omega,\\
\emptyset &\textrm{ otherwise.}
\end{cases}
\end{eqnarray*}
Clearly $\overline{y} \not \in (F + \Delta_{\Omega})(\mathbb{R}^n)$ and $\overline{y}$ is the weak efficient value of the unconstrained optimization problem
\begin{equation*}
\mathrm{minimize}_K \ (F + \Delta_{\Omega})(x) \quad \textrm{ subject to } \quad x \in \mathbb{R}^n.
\end{equation*}
Applying Case~1 to the set-valued mapping $F + \Delta_{\Omega}$ (whose graph is a closed set), we get  a vector $c^* \in K^{+} \setminus \{0\} $ satisfying
\begin{eqnarray} \label{3PT7}
0 &\in& D^*(F + \Delta_{\Omega})(\infty,\overline{y})(c^*).
\end{eqnarray}
On the other hand, we know from Example~\ref{VD36}(iv) that $J(\Delta_{\Omega}) = \{0\}$ and
\begin{eqnarray*}
D^* \Delta_{\Omega}(\infty, 0)(v) &=& N_{\Omega}(\infty) \quad \textrm{ for all } \quad v \in \mathbb{R}^m.
\end{eqnarray*}
Observe that the assumptions in Theorem~\ref{SumeRule} with respect to the set-valued mappings $F$ and $\Delta_{\Omega}$ are satisfied, and so 
\begin{eqnarray*}
D^*(F + \Delta_{\Omega})(\infty,\overline{y})(c^*) & \subset & D^* F(\infty,\overline{y})(c^*) + D^* \Delta_{\Omega}(\infty, 0)(c^*)\\
&=& D^*F(\infty , \overline{y}) (c^*) + N_{\Omega}(\infty).
\end{eqnarray*}
These relations when combined with \eqref{3PT7} yield the necessary optimality condition at infinity: 
\begin{eqnarray*}
0 & \in & D^*F(\infty , \overline{y}) (c^*) + N_{\Omega}(\infty),
\end{eqnarray*}
which is the desired conclusion.
\end{proof}

Theorem~\ref{Fermatatinfinity} is illustrated in the following examples.
 
\begin{example}{\rm 
(i) Consider the set-valued mapping $F \colon \mathbb{R} \rightrightarrows \mathbb{R}$ defined by
\begin{eqnarray*}
F(x) &:=&
\begin{cases}
[e^x, +\infty) & \textrm{ if } x \le 0, \\
\{x^2\} & \textrm{ otherwise.}
\end{cases}
\end{eqnarray*}
Let $\Omega := \mathbb{R}$ and $K := \mathbb{R}_+.$ Observe that $\overline{y} = 0 \in \cl F(\Omega) \setminus F(\Omega)$ is a weak efficient value for \eqref{Problem}. Moreover, a direct calculation shows that $N_{\mathrm{gph} F}(\infty, \overline{y}) = \{(0, -v) \mid v \ge 0\},$ and hence 
\begin{eqnarray*}
D^*F(\infty, \overline{y})(v) &=& 
\begin{cases}
\{0\} & \textrm{ if } v \ge 0, \\
\emptyset & \textrm{ if } v < 0.
\end{cases}
\end{eqnarray*}
Then for $c^* := 1 \in K^+ \setminus \{0\}$ we have $0 \in D^*F(\infty , \overline{y}) (c^*).$

(ii) Consider the set-valued mapping $F \colon \mathbb{R}^2 \rightrightarrows \mathbb{R}^2$ defined by
\begin{eqnarray*}
F(x) &:=& 
\begin{cases}
(x_1, 0) & \textrm{ if } x_1 \ge 0, \\
[1, +\infty)\times [1, +\infty) & \textrm{ otherwise.}
\end{cases}
\end{eqnarray*}
Let $\Omega := \{(x_1, x_2) \in \mathbb{R}^2 \mid x_1^2x_2 \ge 1\}$ and $K := \mathbb{R}^2_{+}.$ It is easy to check that $\overline{y}=(0,0) \in \cl F(\Omega) \setminus F(\Omega)$ is a weak efficient value for \eqref{Problem}. Moreover we have
\begin{eqnarray*}
D^*F(\infty, \overline{y})(0,0) =\{(0, 0)\},  \quad D^*F(\infty, \overline{y})(1,0)  = \{(1, 0)\},
\end{eqnarray*}
and 
\begin{eqnarray*}
N_{\Omega}(\infty) &=& \big(\{0\} \times (-\infty, 0] \big) \cup \big ((-\infty, \infty) \times \{0\} \big).
\end{eqnarray*}
Hence $D^*F(\infty, \overline{y})(0, 0) \cap \big(-N_{\Omega}(\infty)\big) = \{(0, 0)\}$ and so the constraint qualification at infinity holds. For $c^* := (1, 0) \in K^+ \setminus\{(0, 0)\}$ we have 
\begin{eqnarray*}
0 & \in & D^*F(\infty , \overline{y}) (c^*) + N_{\Omega}(\infty). 
\end{eqnarray*}
}\end{example}

\subsection*{Acknowledgments} The first author was supported by the National Research Foundation of Korea Grant funded by the Korean Government (NRF-2019R1A2C1008672). 
A part of this work was done while the second, third and fourth authors were visiting Department of Applied Mathematics, Pukyong National University, Busan, Korea in December 2022 and May 2023. These authors would like to thank the department for hospitality and support during their stay.

\bibliographystyle{abbrv}

\end{document}